\documentclass[11pt]{amsart}

\usepackage[T1]{fontenc}
\usepackage[utf8]{inputenc}
\usepackage{textcomp}
\usepackage{lmodern}
\usepackage{microtype}
\usepackage{xcolor}

\usepackage{amssymb}
\usepackage{amsthm}
\usepackage{amscd}
\usepackage{mathabx}

\usepackage{hyperref}

\usepackage{mathscinet}
\usepackage[giveninits=true,doi=false,url=false,sortcites,backend=biber,backref=true]{biblatex}
\usepackage{bussproofs}

\newtheorem{theorem}{Theorem}[section]

\newtheorem{lemma}[theorem]{Lemma}

\theoremstyle{definition}
\newtheorem{definition}[theorem]{Definition}

\mathchardef\mhyphen="2D

\allowdisplaybreaks[2]

\begin{document}

\title{Polymorphic Ordinal Notations\\{{Work in Progress}}}
\author{Henry Towsner}
\date{\today}
\thanks{Partially supported by NSF grant DMS-2054379}
\address {Department of Mathematics, University of Pennsylvania, 209 South 33rd Street, Philadelphia, PA 19104-6395, USA}
\email{htowsner@math.upenn.edu}
\urladdr{\url{http://www.math.upenn.edu/~htowsner}}

\maketitle

\section{Introduction}

Our goal in this paper is to introduce a new ordinal notation for theories approaching the strength of $\Pi^1_2\mhyphen\mathrm{CA}_0$.\footnote{The current draft represents an in-progress version of this work, missing some exposition and proofs.} Various notations around this strength, and stronger, have been developed by Rathjen and Arai \cite{RathjenStability,RathjenParameterFree,MR1791381,MR3443524,arai2024ordinal}, however we will take a slightly different tactic here.

Our approach is based on the methods of cut-elimination developed by the author in \cite{2403.17922,towsner:MR2499713}. In this approach, one uses the sequent calculus not only to represent proofs, but also to represent functions on proofs. The ordinal terms that result from this naturally include using variables to represent ``ordinal'' bounds on ill-founded deductions.

Here we develop the corresponding ordinal notations (without much reference to the cut-elimination techniques that motivate them). We first present a (more or less) standard presentations of the so-called Buchholz ordinal (the proof-theoretic ordinal of $\Pi^1_1\mhyphen\mathrm{CA}_0$) to recall some typical patterns that appear in such ordinal notations---in particular, a presentation of collapsing functions and the key lemmata they need to satisfy.

The notation we present, as is typical, uses ordinal notations $\Omega_n$, often thought of as representing uncountable cardinals. More formally, our ordinal terms have a distinguished class of countable ordinal terms, and $\Omega_1$ is understood as being larger than any notation representing a countable ordinal term\footnote{Even more formally, our ordinal notation is equipped, at least implicitly, with a system of fundamental sequences. An ordinal of ``countable formal cofinality'' is one in which the associated fundamental sequence of smaller ordinals is indexed by $\mathbb{N}$. $\Omega_1$, on the other hand, is the smallest ordinal whose fundamental sequence is not indexed by $\mathbb{N}$.}, $\Omega_2$ is larger than any ordinal term of ``cardinality equal to $\Omega_2$'', and so on. We then have collapsing functions $\vartheta_n$ which map ordinal terms to terms of smaller cardinality---we take $\vartheta_n\alpha<\Omega_n$ for all $\alpha$, even when $\alpha$ has large cardinality.

As a stepping stone, we give a second presentation of the same ordinal using an approach we call ``polymorphic''---we have a single notation $\Omega$ which takes on the role of different $\Omega_n$ depending on its context in the proof.\footnote{As some limited motivation for this, recall that in the ordinal analysis of $\Pi^1_1\mhyphen\mathrm{CA}_0$ we generally have to ``stratify'' the theory syntactically \cite{MR635488,akiyoshi2017ordinal}, introducing some formalism for keeping track of the way that applications of $\Pi^1_1$-comprehension can be nested, since we need to bound instances of parameter-free $\Pi^1_1$-comprehension by terms involving $\Omega_1$, those involving only parameters which are themselves instantiated with parameter-free $\Pi^1_1$-comprehension by terms involving $\Omega_2$, and so on. The ordinal terms here do not need this stratification---one can bound an instance of $\Pi^1_1$ comprehension with a term involving ``polymorphic $\Omega$'' and then determine the cardinality that $\Omega$ should be interpreted at later.}

Finally, we present new ordinal notations, suitable for theories a bit below parameter-free $\Pi^1_1\mhyphen\mathrm{CA}_0$, which make essential use of this polymorphism.

\section{Background: Buchholz's Ordinal}

We first present an essentially standard notation for Buchholz's ordinal. We include this since there have been many variations on ordinal notations around this strength (see \cite{MR3525540}, for instance), so having a presentation of a familiar system will make it easier to describe the modifications needed in new ordinal notations. Our collapsing function $\vartheta$ is based on the version introduced by Rathjen and Weiermann \cite{rathjen:MR1212407}, motivated in part by the approach to this function in \cite{MR4594953,MR3994445}.

\begin{definition}
  We define the ordinal terms $\mathrm{OT}_{\Omega_\omega}$ together with a distinguished subset $\mathbb{H}$ by:
  \begin{itemize}
  \item if $\{\alpha_0,\ldots,\alpha_{n-1}\}$ is a finite multi-set of elements of $\mathrm{OT}_{\Omega_\omega}$ in $\mathbb{H}$ and $n\neq 1$ then $\#\{\alpha_0,\ldots,\alpha_{n-1}\}$ is in $\mathrm{OT}_{\Omega_\omega}$,
  \item if $\alpha$ is in $\mathrm{OT}_{\Omega_\omega}$ then $\omega^\alpha$ is in $\mathbb{H}$,
  \item for each natural number $n>1$, $\Omega_n$ is in $\mathbb{H}$,
  \item for each $n>1$ and any $\alpha$ in $\mathrm{OT}_{\Omega_\omega}$, $\vartheta_n\alpha$ is an ordinal term in $\mathbb{H}$.
  \end{itemize}

  We define $\mathbb{SC}$ to be all ordinal terms in $\mathbb{H}$ \emph{not} of the form $\omega^\alpha$---that is, $\mathbb{SC}$ consists of ordinal terms of the form $\Omega_n$ or $\vartheta_n\alpha$.
\end{definition}

For reasons we explain later, we choose to use the commutative sum $\#$ as our basic operation instead of the usual $+$, but this choice is not essential. One (trivial) benefit is that we do not need a special case for $0$---$0$ is abbreviates the empty sum $\#\emptyset$.

We adopt the convention that when $\alpha,\beta\in \mathbb{H}$ then $\alpha\#\beta$ means $\#\{\alpha,\beta\}$. We extend this to terms not in $\mathbb{H}$ by collapsing nested applications of $\#$: that is, $\alpha\#\{\beta_0,\ldots,\beta_{n-1}\}$ means $\{\alpha,\beta_0,\ldots,\beta_{n-1}\}$ and so on.

Before defining the ordering, we need to define the critical subterms.
\begin{definition}
  For $n\in\mathbb{N}$, we define $K_n\alpha$ inductively by:
  \begin{itemize}
  \item $K_n\{\alpha_i\}=\bigcup_i K_n\alpha_i$,
  \item $K_n\omega^\alpha=K_n\alpha$,
  \item $K_n\Omega_m=\left\{\begin{array}{ll}
                              \{\Omega_m\}&\text{if }m<n\\
                              \emptyset&\text{if }n\leq m
                                         \end{array}\right.$,
  \item $K_n\vartheta_m\alpha=\left\{\begin{array}{ll}
        K_n\alpha&\text{if }n<m\\
        \{\vartheta_m\alpha\}&\text{if }m\leq n
        \end{array}\right.$
  \end{itemize}
\end{definition}

\begin{definition}
  We define the ordering $\alpha<\beta$ by:
  \begin{itemize}
\item $\#\{\alpha_i\}<\#\{\beta_j\}$ if there is some $\beta_{j_0}\in\{\beta_j\}\setminus\{\alpha_i\}$ such that, for all $\alpha_i\in\{\alpha_i\}\setminus\{\beta_j\}$, $\alpha_i<\beta_{j_0}$,
\item if $\beta\in\mathbb{H}$ then:
  \begin{itemize}
  \item $\beta<\#\{\alpha_i\}$ if there is some $i$ with $\beta\leq\alpha_i$,
  \item $\#\{\alpha_i\}<\beta$ if, for all $i$, $\alpha_i<\beta$,
  \end{itemize}
\item $\omega^\alpha<\omega^\beta$ if $\alpha<\beta$,
\item if $\beta\in\mathbb{SC}$ then:
  \begin{itemize}
  \item $\beta<\omega^\alpha$ if $\beta<\alpha$,
  \item $\omega^\alpha<\beta$ if $\alpha\leq\beta$,
  \end{itemize}
\item $\Omega_m<\Omega_n$ if $m<n$,
\item $\Omega_m<\vartheta_n\alpha$ if there is a $\beta\in K_n\alpha$ with $\Omega_m\leq\beta$,
\item $\vartheta_m\alpha<\Omega_n$ if for all $\beta\in K_m\alpha$, $\beta<\Omega_n$,
\item $\vartheta_m\alpha<\vartheta_n\beta$ if:
  \begin{itemize}
  \item there is some $\gamma\in K_n\beta$ so that $\vartheta_n\alpha\leq\gamma$,
  \item for all $\gamma\in K_m\alpha$, $\gamma<\vartheta_n\beta$, and either:
    \begin{itemize}
    \item $m<n$, or
    \item $m=n$ and $\alpha<\beta$.
    \end{itemize}
  \end{itemize}
\end{itemize}
\end{definition}

We generally expect $\vartheta_3\alpha$ to have cardinality $\Omega_2$. When defining an expression like $\vartheta_3\Omega_1$, where $\alpha$ has cardinality less than $\Omega_2$, we have to decide whether this guideline applies even when $\alpha$ itself is small, or if $\vartheta_3\Omega_1$ should instead have cardinality comparable to $\Omega_1$. We take the latter path, even though it is less conventional and a bit less elegant here, since it better illustrates some behavior we wish to emphasize. (It would be only a minor change to the definition of the ordering to switch convention, of course.)

It is customary to analogize $\Omega_1$ to an uncountable ordinal; then ordinals like $\vartheta_1\Omega_1$ are ``countable'', and those like $\Omega_3\#\Omega_2\#\omega^{\Omega_1}$ has ``cardinality $\aleph_3$'' and so on.

More formally, we may assign to each ordinal term a \emph{formal cardinality}, which we may take to be a natural number in $\mathbb{N}$---the formal cardinality $0$ is the countable ordinals, $1$ is those at the cardinality of $\Omega_1$, and so on. It is convenient to first define the set of formal cardinalities appearing in $\alpha$.

\begin{definition}
  We define $\mathbb{FC}(\alpha)$ by:
  \begin{itemize}
  \item $\mathbb{FC}(\#\{\alpha_i\})=\bigcup_i\mathbb{FC}(\alpha_i)$,
  \item $\mathbb{FC}(\omega^\alpha)=\mathbb{FC}(\alpha)$,
  \item $\mathbb{FC}(\Omega_n)=\{n\}$,
  \item $\mathbb{FC}(\vartheta_n\alpha)=\mathbb{FC}(\alpha)\setminus\{m\mid m\geq n\}$.
  \end{itemize}
  
  We then define $\overline{\mathbb{FC}}(\alpha)=\max\mathbb{FC}(\alpha)$ (where the maximum of the empty set is $0$.)
\end{definition}

It is easy to check that $\overline{\mathbb{FC}}(\alpha)<\overline{\mathbb{FC}}(\beta)$ implies $\alpha<\beta$.

There is a second analogy, less often discussed, which is that we could think of $\Omega_n$ as being something akin to a free variable, with $\vartheta_n$ acting as the corresponding quantifier. The ordinals of formal cardinality $0$ are precisely the ``closed'' ordinals. %

It will be useful to consider ordinal terms with conventional free variables\footnote{This is why we prefer to use $\#$: substitution of variables, at least by ordinals in $\mathbb{H}$, does not alter the normal of an ordinal when we use $\#$.}: we add, for each $n>0$, an additional ordinal term $v_n$, understood to have formal cardinality $n$, together with the rule that $\vartheta_n\alpha$ cannot contain any $v_m$ with $m\geq n$. We extend the ordering by specifying that $v_n$ is in $\mathbb{SC}$, is smaller than $\Omega_n$, but is incomparable to other terms; we also set $K_nv_m=\left\{\begin{array}{ll}
                              \{v_m\}&\text{if }m<n\\
                              \emptyset&\text{if }n\leq m
                                         \end{array}\right.$, just like $\Omega_m$. (Only the first case actually gets applied, because of the restriction on variables appearing in $\vartheta_n\alpha$.)

For cut-elimination proofs of the sort in \cite{2403.17922}, the key properties of these ordinal terms are given by the following lemma.

\begin{lemma}[Key Lemma, weak version]
  \begin{enumerate}
  \item If $\alpha<\beta$ and $\gamma$ has formal cardinality $<n$ then $\alpha[v_n\mapsto\gamma]<\beta[v_n\mapsto\gamma]$.
  \item If $\alpha<\beta$ are ordinal terms containing no free variables $v_m$ with $m> n$ and $K_n\alpha<\vartheta_n\beta$ then $\vartheta_n\alpha<\vartheta_n\beta$.
  \item If $\alpha<\beta$ are ordinal terms containing no free variables $v_m$ with $m\geq n$, $\gamma<\beta$ is an ordinal term containing no free variables $v_m$ with $m>n$, $K_n\alpha<\vartheta_n\beta$, and $K_n\gamma<\vartheta_n\beta$, then $\vartheta_n(\gamma[v\mapsto\vartheta_n\alpha])<\vartheta_n\beta$.
\end{enumerate}
\end{lemma}
In fact, we need slightly more than this---we need a ``relativized'' version in which the condition $K_n\alpha<\beta$ is replaced by $K_n\alpha<\max\{\beta,\delta\}$ for some $\delta$, and corresponding changes are made elsewhere.

\begin{definition}
  When $\overline{\mathbb{FC}}(\gamma)<n$, we define $D_{n,\gamma}\beta=\vartheta_n(\omega^{\Omega_n\#\beta}\#\gamma)$.  More generally, for $m\in(0,n)$, we define $D_{m,\gamma}\beta=D_{m,0}D_{m+1,\gamma}\beta$.
  
  When $\overline{\mathbb{FC}}(\gamma)<n$, we define $\alpha\ll^n_\gamma\beta$ to hold if $\alpha<\beta$ and, for all $m\leq n$, every element of $K_m\alpha$ is bounded by $D_m\beta$.
\end{definition}
Note that this implies that $\ll^0_\gamma$ is just $<$.

\begin{lemma}[Key Lemma]
  \begin{enumerate}
  \item If $\alpha<\beta$ and $\delta$ has formal cardinality $<n$ then $\alpha[v_n\mapsto\gamma]<\beta[v_n\mapsto\gamma]$.
  \item If $\overline{\mathbb{FC}}(\delta)<n$, $\alpha\ll^n_\delta\beta$ are ordinal terms containing no free variables $v_m$ with $m\geq n$ then $D_{n,\delta}(\alpha)\ll^{n-1}_0D_{n,\delta}(\beta)$.
  \item If $\overline{\mathbb{FC}}(\delta)<n$, $\alpha\ll^n_\delta\beta$ are ordinal terms containing no free variables $v_m$ with $m\geq n$, and $\gamma\ll^n_\delta\beta$ is an oridnal term containing no free variables $v_m$ with $m>n$, then $D_{n,D_{n,\delta}(\alpha)}(\gamma[v_n\mapsto D_{n,\delta}(\alpha)])\ll^{n-1}_{\delta}D_{n,\gamma}(\beta)$.
  \end{enumerate}
\end{lemma}
\begin{proof}
  \begin{enumerate}
  \item   By a straightforward induction on $\alpha$ and $\beta$. The most difficult case is when $\alpha=\vartheta_m\alpha'$ and $\beta=\vartheta_{m'}\beta'$, if $m'\leq n$ then $v_n$ does not appear $\beta'$, so $\beta$ is unchanged after substitution. Also we cannot have $v_n$ in $\alpha'$ (if we did, we would have $m>n$ and therefore $v_n\in K_n\alpha'$, so $\vartheta_{m'}\beta'<\vartheta_m\alpha'$), so $\vartheta_m\alpha'$ is also unchanged by substitution, so the comparison remains.

    If $m'>n$ then the claim follows immediately from the inductive hypothesis.
  \item First, observe that $\omega^{\Omega_n\#\alpha}\#\delta<\omega^{\Omega_n\#\beta}\#\delta$. Also, $K_n(\omega^{\Omega_n\#\alpha}\#\delta)=K_n\alpha\cup K_n\beta$, so any element of $K_n(\omega^{\Omega_n\#\alpha}\#\delta)$ is bounded by an element f either $K_n\beta$ or $K_n\delta$, and is therefore $<D_{n,\delta}(\beta)$. Therefore $D_{n,\delta}(\alpha)<D_{n,\delta}(\beta)$.

    Finally, if $m<n$ and $\eta\in K_m D_{n,\delta}(\alpha)$ then $\eta\in K_m\alpha$ or $K_m\delta$, so in either case $\eta$ is bounded $D_{m,\delta}(\beta)$.
  \item Let $\gamma'=\gamma[v_n\mapsto D_{n,\delta}(\alpha)]$. We have $\gamma'<\beta$, so also $\omega^{\Omega_n\#\gamma'}\# D_{n,\delta}(\alpha)<\omega^{\Omega_n\#\beta}\#\delta$. We have $K_n \omega^{\Omega_n\#\gamma'}\# D_{n,\delta}(\alpha)\subseteq K_n\gamma\cup K_n\delta \cup K_n\alpha$, so in particular every element is bounded by an element of $K_n\beta\cup K_n\delta$, so $D_{n,D_{n,\delta}(\alpha)}(\gamma')<D_{n,\delta}(\beta)$.

    Similarly, if $m<n$ and $\eta\in K_m D_{n,D_{n,\delta}}(\gamma')$ then $\eta\in K_m\gamma\cup K_m\delta\cup K_m\alpha\cup\{D_{n,\delta}(\alpha)\}$ and is therefore bounded by $D_{m,\delta}(\beta)$.
  \end{enumerate}
\end{proof}

\section{Buchholz's Ordinal, Polymorphically}

\subsection{Ordinal Terms}

The presence of terms with variables invites us to consider the term $\vartheta_1(\Omega_1\#v_1)$ as giving a function on ordinals of formal cardinality $0$, mapping $\alpha$ to $f(\alpha)=\vartheta_1(\Omega_1\#\alpha)$.

We would like to extend this function to ordinals of larger formal cardinality in the following way: we would like to take $f(\Omega_1)=\vartheta_2(\Omega_2\#\Omega_1)$. That is, we would like to interpret $\Omega_1$ ``polymorphically'' so that it always means ``the next cardinal''.\footnote{The approach here is clearly closely related to the understanding of the Howard--Bachmann ordinal in terms of dilators \cite{MR656793,girard_logical_complexity,MR3994445}.}

With this modification, we no longer want the family of terms $\Omega_n$ with corresponding collapsing functions $\vartheta_n$; instead we want a single term, $\Omega$, with a single collapsing function $\vartheta$.

However we need to write terms like the result of the substitution $(\vartheta(\Omega\#v))[v\mapsto\Omega]$; this cannot be $\vartheta(\Omega\#\Omega)$, since the two $\Omega$'s are different. Our analogy to free variables helps us here: we think of $\Omega$ as a free variable, $\vartheta$ as the corresponding binder, and we use de Bruijn indices to distinguish different ``levels'' of $\Omega$. For reasons that will become apparent, it will be helpful to have our levels go down rather than up, so we will write $\Omega^{(J)}$, for various integers $J\leq 0$, to help us distinguish different copies\footnote{It feels a trifle silly to trumpet how we're going to remove the subscripts and then right away replace them with superscripts, but we have actually made a meaningful change here, and the superscripts will behave rather differently.} of $\Omega$.

Formally, $-(J-1)$ counts the number of applications of $\vartheta$ it takes to bind $\Omega$. We will sometimes use color coding to match applications of $\vartheta$ with the corresponding bound copies of $\Omega$ (and write unbound copies of $\Omega$ in black). So ${\color{blue}\vartheta}({\color{blue}\Omega}^{(0)}\#v)$ is the function we were considering above, and ${\color{blue}\vartheta}({\color{blue}\Omega}^{(0)}\#v)[v\mapsto\Omega^{(0)}]={\color{blue}\vartheta}({\color{blue}\Omega}^{(0)}\#\Omega^{(-1)})$. (One confusing aspect of de Bruijn indices, which takes some getting used to, is that the same variable is written differently in different contexts. For instance, in $\Omega^{(0)}\#{\color{blue}\vartheta}({\color{blue}\Omega}^{(0)}\#\Omega^{(-1)})$, the two black copies of $\Omega$ are \emph{the same}, even though the index has changed: moving inside the $\vartheta$ has decremented the index.)

We can restate the ordinal notation for Buchholz's ordinal using this approach.

\begin{definition}
  We define the ordinal terms $\mathrm{OT}^{poly}_{\Omega_\omega}$ together with a distinguished subset $\mathbb{H}$ by:
  \begin{itemize}
  \item if $\{\alpha_0,\ldots,\alpha_{n-1}\}$ is a finite multi-set of elements of $\mathrm{OT}^{poly}_{\Omega_\omega}$ in $\mathbb{H}$ and $n\neq 1$ then $\#\{\alpha_0,\ldots,\alpha_{n-1}\}$ is in $\mathrm{OT}^{poly}_{\Omega_\omega}$,
  \item if $\alpha$ is in $\mathrm{OT}^{poly}_{\Omega_\omega}$ then $\omega^\alpha$ is in $\mathbb{H}$,
  \item for each integer $J\leq 0$, $\Omega^{(J)}$ is in $\mathbb{H}$,
  \item for any $\alpha$ in $\mathrm{OT}^{poly}_{\Omega_\omega}$, $\vartheta\alpha$ is in $\mathbb{H}$.
  \end{itemize}

  We define $\mathbb{SC}$ to be all ordinal terms in $\mathbb{H}$ \emph{not} of the form $\omega^\alpha$.
\end{definition}

We next need to define the ordering. As a first step, how should we compare $\Omega^{(0)}$ and $\Omega^{(-1)}$? In the term ${\color{blue}\vartheta}({\color{blue}\Omega}^{(0)}\#\Omega^{(-1)})$, $\Omega^{(0)}$ is ``bound'' (that is, has been collapsed by a $\vartheta$) while $\Omega^{(-1)}$ is ``free''. We always collapse larger cardinals first, so we must have $\Omega^{(0)}>\Omega^{(-1)}$. The comparison should depend only on the order of the indices, so we must have $\Omega^{(J)}<\Omega^{(J')}$ when $J<J'$. (We have chosen to use non-positive indices precisely to make the map $J\mapsto\Omega^{(J)}$ order-preserving rather than order reversing.)

This immediately, of course, makes the full system ill-founded. We will resolve this below by normalizing terms based on the smallest cardinality appearing.

It is helpful to think of formal cardinality in this system as being a relative notion---$\Omega^{(0)}$ means \emph{the current cardinality}, $\Omega^{(-1)}$ means the cardinality one smaller, and so on. This will come up repeatedly in the definitions below: we will set out to define some notion, but in order to make the induction go through, we will have to keep track of how many times we have gone inside a $\vartheta$.

To begin with, let us define the notion of the formal cardinality of a term. We choose to preserve the ordering from the previous section, so large cardinalities are large; this means our cardinalities will normally be negative numbers, and we will take $-\infty$ to be the countable cardinality.
\begin{definition}
  We define $\mathbb{FC}^{\leq J}(\alpha)$ inductively by:
  \begin{itemize}
  \item $\mathbb{FC}^{\leq J}(\#\{\alpha_i\})=\bigcup_i\mathbb{FC}^{\leq J}(\alpha_i)$,
  \item $\mathbb{FC}^{\leq J}(\omega^\alpha)=\mathbb{FC}^{\leq J}(\alpha)$,
  \item $\mathbb{FC}^{\leq J}(\Omega^{(J')})=\left\{\begin{array}{ll}
                                                     \{J'-J\}&\text{if }J'\leq J\\
                                                      \emptyset&\text{if }J'>J
                                                    \end{array}\right.$
  \item $\mathbb{FC}^{\leq J}(\vartheta\alpha)=\mathbb{FC}^{\leq J-1}\alpha$.
  \end{itemize}

  We define $\overline{\mathbb{FC}}^{\leq J}(\alpha)=\sup\mathbb{FC}^{\leq J}(\alpha)$ where $\sup\emptyset=-\infty$. %
\end{definition}
For instance, $\overline{\mathbb{FC}}(\Omega^{(-1)}\#\Omega^{(-2)}\#\Omega^{(-2)})=-1$---the largest ``cardinal'' appearing is $\Omega^{(-1)}$. We should think of this as saying that this term has cardinality one smaller than the cardinality of our ``starting point''. Similarly, $\overline{\mathbb{FC}}({\color{blue}\vartheta}({\color{blue}\Omega}^{(0)}\#\Omega^{(-1)}))=0$. That is, $\mathbb{FC}$ is telling us which uncollapsed cardinals appearing inside $\alpha$, with their levels reinterpreted to be relative to our starting point.

The next thing we need to define is the critical subterms $K\alpha$. Consider the term ${\color{red}\vartheta}{\color{blue}\vartheta}({\color{blue}\Omega}^{(0)}\#{\color{red}\Omega}^{(-1)})$. This is the term we would previous have written as $\vartheta_1\vartheta_2(\Omega_2\#\Omega_1)$, so it should have no critical subterms. On the other hand, ${\color{red}\vartheta}{\color{blue}\vartheta}({\color{blue}\Omega}^{(0)})$ is analogous to $\vartheta_1\vartheta_1\Omega_1$, so does have a critical subterm ${\color{blue}\vartheta}{\color{blue}\Omega}^{(0)}$.

How does $K$ know the difference? The answer is that $K$ knows that it is trying to collect terms $<\Omega^{(0)}$, and so should only pick up subterms which are themselves $<\Omega^{(0)}$. As we define this inductively, our ``position'' in the relative hierarchy of cardinals changes: say $\vartheta\alpha$ has cardinality too large, so $K\vartheta\alpha$ is supposed to be $K\alpha$; moving inside the $\vartheta$ shifts our position in the cardinals, so the thing we called $\Omega^{(0)}$ outside of $\vartheta\alpha$ is the thing called $\Omega^{(1)}$ inside $\alpha$.

There is a further complication. Consider the term ${\color{red}\vartheta}({\color{blue}\vartheta}({\color{blue}\Omega}^{(0)}\#\Omega^{(-2)})\#{\color{red}\Omega}^{(-1)})$. This is analogous\footnote{We should highlight here that the notation given by $\mathrm{OT}^{poly}_{\Omega_\omega}$ is \emph{not} trivially isomorphic to $\mathrm{OT}_{\Omega_\omega}$.} to $\vartheta_2(\vartheta_2(\Omega_2\#\Omega_1)\#\Omega_2)$---in particular, we should have $K({\color{blue}\vartheta}({\color{blue}\Omega}^{(0)}\#\Omega^{(-2)})\#{\color{red}\Omega}^{(-1)})=\{{\color{blue}\vartheta}({\color{blue}\Omega}^{(0)}\#\Omega^{(-1)})\}$---that is, we ought to shift the index down to reflect that it is no longer inside ${\color{red}\vartheta}$.

So we need to define the notion of shifting the indices in a term up or down to represent ``the same term from a different perspective''.\footnote{All this shifting might lead one to be skeptical of the entire ``relative cardinality'' approach. We could instead have tried to define cardinality in an absolute way, so that $\Omega^{(0)}$ means the first cardinality wherever it appears. This would change when shifting has to be done, but would not save us from doing it---for instance, we would have ${\color{blue}\vartheta_0}({\color{blue}\Omega}^{(0)}\#v)[v\mapsto\Omega^{(0)}]={\color{blue}\vartheta_1}({\color{blue}\Omega}^{(-1)}\#\Omega^{(0)})$.}

We once again need to include a $\cdot^{\leq J}$ restriction---we should only shift unbound cardinals, leaving bound ones alone.
\begin{definition}
  We define $\alpha^{\leq J}_{\pm n}$ by:
  \begin{itemize}
  \item $\#\{\alpha_i\}^{\leq J}_{\pm n}=\#\{(\alpha_i)^{\leq J}_{\pm n}\}$,
  \item $(\omega^\alpha)^{\leq J}_{\pm n}=\omega^{\alpha^{\leq J}_{\pm n}}$,
  \item $(\Omega^{(J')})^{\leq J}_{\pm n}=\left\{\begin{array}{ll}
                                                   \Omega^{(J'\pm n)}&\text{if }J'\leq J\\
                                                   \Omega^{(J')}&\text{if }J'>J
                                                 \end{array}\right.$
  \item $(\vartheta\alpha)^{\leq J}_{\pm n}=\vartheta(\alpha^{\leq J-1}_{\pm n})$.                 \end{itemize}
\end{definition}

Note that $\alpha^{\leq J}_{+n}$ may not be well-behaved unless $\overline{\mathbb{FC}}^{\leq J}(\alpha)\leq -n$---for instance, we cannot take $(\vartheta\Omega^{(-1)})^{\leq 0}_{+1}$, since increasing the index of $\Omega^{(-1)}$ by $1$ would create a collision. (Maybe more directly, we cannot take $(\Omega^{(0)})^{\leq 0}_{+1}$ because $\Omega^{(1)}$ is not a term.)

\begin{definition}
  We define $K^{< J}\alpha$ inductively by:
  \begin{itemize}
  \item $K^{< J}\#\{\alpha_i\}=\bigcup_i K^{< J}\alpha_i$,
  \item $K^{< J}\omega^\alpha=K^{< J}\alpha$,
  \item $K^{< J}\Omega^{(J')}=\left\{\begin{array}{ll}
                                          \{\Omega^{(J'-(J-1))}\}&\text{if }J'< J\\
                                          \emptyset&\text{if }J'\geq J
                                        \end{array}\right.$,
  \item $K^{< J}\vartheta\alpha=\left\{\begin{array}{ll}
                                            (\vartheta\alpha)^{\leq 0}_{+(1-J)}&\text{if }\overline{\mathbb{FC}}^{\leq 0}(\vartheta\alpha)<J\\
                                            K^{< J-1}\alpha&\text{if }\overline{\mathbb{FC}}^{\leq 0}(\vartheta\alpha)\geq J
                                          \end{array}\right.$
  \end{itemize}
\end{definition}
To help motivate the numbers in the definition above, keep in mind that we are interested in using $K^{<0}\alpha$ to help us understand $\vartheta\alpha$. $K^{<J}\beta$ means we are looking at some subterm of $\alpha$ which is inside $-J$ applications of $\vartheta$, and when we pull subterms into the set $K^{<J}\beta$, we wish to shift them by enough so that we can correctly compare them to the term $\vartheta\alpha$.

Having adjusted the definitions, the definition of the order is almost unchanged.
\begin{definition}
  We define the ordering $\alpha<\beta$ by:
  \begin{itemize}
\item $\#\{\alpha_i\}<\#\{\beta_j\}$ if there is some $\beta_{j_0}\in\{\beta_j\}\setminus\{\alpha_i\}$ such that, for all $\alpha_i\in\{\alpha_i\}\setminus\{\beta_j\}$, $\alpha_i<\beta_{j_0}$,
\item if $\beta\in\mathbb{H}$ then:
  \begin{itemize}
  \item $\beta<\#\{\alpha_i\}$ if there is some $i$ with $\beta\leq\alpha_i$,
  \item $\#\{\alpha_i\}<\beta$ if, for all $i$, $\alpha_i<\beta$,
  \end{itemize}
\item $\omega^\alpha<\omega^\beta$ if $\alpha<\beta$,
\item if $\beta\in\mathbb{SC}$ then:
  \begin{itemize}
  \item $\beta<\omega^\alpha$ if $\beta<\alpha$,
  \item $\omega^\alpha<\beta$ if $\alpha\leq\beta$,
  \end{itemize}
\item $\Omega^{(J)}<\Omega^{(J')}$ if $J<J'$,
\item $\Omega^{(J)}<\vartheta\alpha$ if there is a $\beta\in K^{< 0}\alpha$ with $\Omega^{(J)}\leq\beta$,
\item $\vartheta\alpha<\Omega^{(J)}$ if for all $\beta\in K^{< 0}\alpha$, $\beta<\Omega^{(J)}$,
\item $\vartheta\alpha<\vartheta\beta$ if:
  \begin{itemize}
  \item $\alpha<\beta$ and, for all $\gamma\in K^{< 0}\alpha$, $\gamma<\vartheta\beta$, or
  \item $\beta<\alpha$ and there is some $\gamma\in K^{< 0}\beta$ so that $\vartheta\alpha\leq\gamma$.
  \end{itemize}
\end{itemize}
\end{definition}

\subsection{Well-Foundedness}

Since $\Omega^{(0)}>\Omega^{(-1)}>\cdots$, the ordinal terms are certainly not well-founded on their face. This is because our notation looks at ordinal terms ``from above'', fixing $\Omega^{(0)}$ to be the highest cardinality. We should instead compare ``from below'', demanding the ordinals share the same ground.

\begin{definition}
  For any $\alpha$, we define $\mathbb{G}(\alpha)=\min\mathbb{FC}(\alpha)$ or $\{-\infty\}$ if $\mathbb{FC}(\alpha)=\emptyset$.
\end{definition}
We call $\mathbb{G}(\alpha)$ the \emph{ground} of $\alpha$---it is the lowest cardinality appearing free, so we think of it as the ``base level'' of the ordinal. We would like to compare ordinals by positioning them to share the same ground.

Given an ordinal $\alpha$, we write $\alpha^*=\left\{\begin{array}{ll}\alpha&\text{if }\overline{\mathbb{FC}}(\alpha)=-\infty\\\alpha^{\geq 0}_{-\overline{\mathbb{FC}}(\alpha)}&\text{otherwise}\end{array}\right.$. We may think of this as a normalized version of $\alpha$, in the sense that we have arranged pushed the cardinalities up as high as we are allowed to; in particular, $\overline{\mathbb{FC}}(\alpha^*)\in\{-\infty,0\}$.

We could think of ordinal terms with $\overline{\mathbb{FC}}(\alpha)=0$ and $\mathbb{G}(\alpha)\geq n$ as being the terms of cardinality $\Omega_{n+1}$ by interpreting $\Omega^{(i)}$ as $\Omega_{n+i+1}$.

We will define two collections of sets $\mathrm{Acc}_n\subseteq M_n$ for $n\in\{-\infty\}\cup \mathbb{N}$.
\begin{definition}
  We let $M_{-\infty}$ be all ordinal terms of formal cardinality $-\infty$. Given $M_n$, we let $\mathrm{Acc}_n$ be the well-founded part of $M_n$.

  For $n\in\mathbb{N}$, given $\mathrm{Acc}_i$ for $i<n$, we set $M_n$ to be the set of ordinal terms $\alpha$ with $\overline{\mathbb{FC}}(\alpha)=0$, $\mathbb{G}(\alpha)\geq -n$, and $\{\beta^*\mid \beta\in K^{< 0}\alpha\}\subseteq\bigcup_{i<n}\mathrm{Acc}_i$.
\end{definition}
Note that when $\beta\in K^{< 0}\alpha$, we always have $\overline{\mathbb{FC}}(\beta)<\overline{\mathbb{FC}}(\alpha)$ and therefore $\mathbb{G}(\beta^*)>\mathbb{G}(\alpha)$.%

The construction of $\mathrm{Acc}_n$ requires an application of $\Pi^1_1$ comprehension using $\bigcup_{i<n}\mathrm{Acc}_n$ as a parameter. In particular, contructing the whole collection of sets $\mathrm{Acc}_n$ for all $n$ cannot be done in $\Pi^1_1\mhyphen\mathrm{CA}_0$.

\begin{lemma}
  Whenever $m\leq n$, $\alpha\in\mathrm{Acc}_m$, and $\beta\in\mathrm{Acc}_n$, we have $\alpha^{\geq 0}_{-(n-m)}\#\beta\in\mathrm{Acc}_n$.
\end{lemma}
\begin{proof}
  We proceed first by induction on $m$, then on $\alpha$, and then by induction on $\beta$. From the definitions, it is easy to see that $\alpha^{\geq 0}_{-(n-m)}\#\beta\in M_n$.

  It suffices to show that, for any $\gamma<\alpha^{\geq 0}_{-(n-m)}\#\beta$ in $M_n$, $\gamma\in\mathrm{Acc}_n$.

  So let $\gamma<\alpha^{\geq 0}_{-(n-m)}\#\beta$ be given. We consider cases for $\gamma$. If $\gamma\in\mathbb{H}$ then either $\gamma\leq\alpha^{\geq 0}_{-(n-m)}$ or $\gamma\leq\beta$. In the latter case, $\gamma\in\mathrm{Acc}_n$ since $\beta$ is. If $\gamma\leq\alpha^{\geq 0}_{-(n-m)}$ then, since $\overline{\mathbb{FC}}(\gamma)\in\{-\infty,0\}$, we must have $m=n$, so $\alpha^{\geq 0}_{-(n-m)}=\alpha$, so $\gamma\in\mathrm{Acc}_n$.

  In the remaining case, $\gamma=\#\{\gamma_i\}$. If $\gamma\leq\beta$ then we immediately have $\gamma\in\mathrm{Acc}_n$, so suppose $\beta<\gamma<\alpha^{\geq 0}_{-(n-m)}\#\beta$. Let $\gamma'=\#\{\gamma_i\mid \alpha^{\geq 0}_{-(n-m)}\leq\gamma_i\}$---that is, the sum where we remove all components $<\alpha^{\geq 0}_{-(n-m)}$.

  We have $\gamma'\leq\gamma<\alpha^{\geq 0}_{-(n-m)}\#\beta$. Suppose $\overline{\mathbb{FC}}(\gamma')=\overline{\mathbb{FC}}(\gamma)$. Then, since $\gamma'\leq\beta$, we have $\gamma'\in\mathrm{Acc}_n$. We may then repeatedly apply the main inductive hypothesis to each $\gamma_i<\alpha^{\geq 0}_{-(n-m)}$ to show that $\gamma\in\mathrm{Acc}_n$ as well.

  If $\overline{\mathbb{FC}}(\gamma')<\overline{\mathbb{FC}}(\gamma)$ then we must have $m=n$ and $\gamma'=0$ and the result follows from $\alpha^{\geq 0}_{-(n-m)}=\alpha\in\mathrm{Acc}_n$ and the main inductive hypothesis.
\end{proof}

\begin{lemma}
  If $\alpha\in\mathrm{Acc}_n$ then $\omega^\alpha\in\mathrm{Acc}_n$.
\end{lemma}
\begin{proof}
  By induction on $n$ and then on $\alpha$, it suffices to show that for all $\beta<\omega^\alpha$ with $\beta\in M_n$, $\beta\in\mathrm{Acc}_n$.

  We proceed by structural induction. If $\beta=\#\{\beta_i\}$ then each $\beta_i<\omega^\alpha$ and therefore each $\beta_i\in\mathrm{Acc}_n$ (by the side inductive hypothesis if $\overline{FC}(\beta_i)=n$ and by the definition of $M_n$ and the main inductive hypothesis if $\overline{FC}(\beta_i)<n$). Since $\bigcup_i\mathrm{Acc}_i$ is closed under sums by th eprevious lemma, $\beta\in\mathrm{Acc}_n$.

  If $\beta=\omega^{\beta'}$ then $\beta'<\alpha$ and $\beta\in\mathrm{Acc}_n$ by the main inductive hypothesis.

  If $\beta=\vartheta\beta'$ then $\vartheta\beta'\leq\alpha$ and therefore $\vartheta\beta'\in\mathrm{Acc}_n$.
\end{proof}

\begin{lemma}
  If $\alpha\in\mathrm{Acc}_n$ then $\vartheta(\alpha^{\leq 0}_{-1})\in\mathrm{Acc}_n$.
\end{lemma}
\begin{proof}
  This includes the case where $n=-\infty$ (in which case $\alpha^{\leq 0}_{-1}=\alpha$)---this is the case where we do not actually bind any copies of $\Omega$.

  We proceed by main induction on $\alpha$. It suffices to show that, for all $\gamma<\vartheta\alpha^{\geq 0}_{-1}$ with $\gamma\in M_n$, $\gamma\in\mathrm{Acc}_n$. We show this by a structural induction on $\gamma$.

  If $\gamma$ is $\Omega^{(0)}$ then this is immediate, since in this case $\Omega^{(0)}$ is the smallest element of $M_n$. If $\gamma=\#\{\gamma_i\}$ then each $\gamma_i$ in $M_n$ is less than $\vartheta\alpha^{\geq 0}_{-1}$ and thefore in $\mathrm{Acc}_n$. If $\gamma_i$ is in $M_m$ for some $m<n$ then, since $\gamma\in M_n$, we always have $\gamma_i\in\mathrm{Acc}_m$. Then by the lemma above, $\gamma\in\mathrm{Acc}_n$.

  If $\gamma=\omega^{\gamma'}$ then $\gamma'<\vartheta\alpha^{\geq 0}_{-1}$, so $\gamma'\in\mathrm{Acc}_n$, so $\gamma\in\mathrm{Acc}_n$.

  So suppose $\gamma=\vartheta\beta$. If $\beta<\alpha^{\geq 0}_{-1}$ then $\overline{\mathbb{FC}}(\beta)<0$, so $\beta$ is precisely of the form $\delta^{\geq 0}_{-1}$ for some $\delta$ with $\overline{\mathbb{FC}}^{\geq 0}(\delta)\in\{-\infty,0\}$. We wish to show that $\delta\in M_n$, so consider some $\delta'\in K^{< 0}\delta$. Then we have $(\delta')^{\geq 0}_{-1}\in K^{< 0}\beta$, so $(\delta')^{\geq 0}_{-1}<\vartheta\alpha^{\geq 0}_{-1}$. If $\overline{\mathbb{FC}}^{\geq 0}(\delta')=\overline{\mathbb{FC}}(\gamma)$ then the inductive hypothesis applies and $((\delta')^{\geq 0}_{-1})^*\in\mathrm{Acc}_n$. Otherwise $\overline{\mathbb{FC}}^{\geq 0}(\delta')<\overline{\mathbb{FC}}(\gamma)$ and, so there is $\delta''\in K^{< 0}\gamma$ with $(\delta'')^*=(\delta')^*$, and so since $\gamma\in M_n$, $(\delta')^*\in\mathrm{Acc}_m$ for the correct $m$. It follows that $\delta\in M_n$, so by the inductive hypothesis, $\gamma=\vartheta\delta^{\geq 0}_{-1}\in\mathrm{Acc}_n$.

  Otherwise, suppose $\alpha^{\geq 0}_{-1}<\beta$. Then we must have some $\delta\in K^{< 0}\alpha^{\geq 0}_{-1}$ so that $\vartheta\beta\leq\delta$. We must have $\overline{\mathbb{FC}}(\delta)<0$, and since $\gamma\in M_n$, it follows that $\overline{\mathbb{FC}}(\delta)=\overline{\mathbb{FC}}(\gamma)=-\infty$, so $n=-\infty$. It follows that $\delta\leq\alpha^{\geq 0}_{-1}=\alpha$, and therefore $\delta\in\mathrm{Acc}_{n}$.
\end{proof}

\begin{lemma}
  If $\alpha\in\mathrm{Acc}_n$ and $n>-\infty$ then $(\vartheta\alpha)^*\in\bigcup_{m<n}\mathrm{Acc}_m$.
\end{lemma}
\begin{proof}
We focus on $\alpha$ with either $\{0,-1\}\subseteq\mathbb{FC}(\alpha)$ or $\mathbb{FC}(\alpha)=\{0\}$. (If $0\not\in\mathbb{FC}(\alpha)$, we are covered by the previous lemma. If $-1\not\in\mathbb{FC}(\alpha)$ but $\overline{\mathbb{FC}}(\vartheta\alpha)>-\infty$ then there is an $\alpha'$ with $\{0,-1\}\subseteq\mathbb{FC}(\alpha')$ and $(\vartheta\alpha')^*=(\vartheta\alpha)^*$.)

We proceed by induction on $\alpha\in\mathrm{Acc}_n$ with this property. Since $\mathbb{G}(\vartheta\alpha)<\mathbb{G}(\alpha)$, we certainly have $(\vartheta\alpha)^*\in\bigcup_{m<n}M_m$. Let $m=\mathbb{G}((\vartheta\alpha)^*)$.
  
  It suffices to show that for all $\gamma<(\vartheta\alpha)^*$ in $M_m$, $\gamma\in\mathrm{Acc}_m$. We show this by a side structural induction on $\gamma$.

  If $\gamma$ is $\Omega^{(0)}$, this is immediate since $\Omega^{(0)}$ is the smallest element of formal cardinality $0$. If $\gamma=\#\{\gamma_i\}$ then each $\gamma_i$ is either in $\mathrm{Acc}_k$ for some $k<m$, or $\gamma_i\in M_m$, $\gamma_i<(\vartheta\alpha)^*$, and therefore $\gamma_i\in\mathrm{Acc}_m$ by the side inductive hypothesis. Then, by closure under addition, $\gamma\in\mathrm{Acc}_m$.

  Similarly, if $\gamma=\omega^{\gamma'}$ then $\gamma'<(\vartheta\alpha)^*$, so $\gamma'\in\mathrm{Acc}_m$, and therefore $\omega^{\gamma'}\in\mathrm{Acc}_m$.

  For the main case, suppose $\gamma=\vartheta\gamma'$. If $\overline{\mathbb{FC}}(\gamma')<0$ then, by the previous lemma, $\gamma\in\mathrm{Acc}_m$. Otherwise, since $\vartheta\gamma'\in M_m$, we must have either $\overline{\mathbb{FC}}(\vartheta\gamma')=-\infty$ or $-1\in \mathbb{FC}(\gamma')$. If $\delta\in K^{< 0}\gamma'$ then either $\delta\in K^{< 0}\gamma$ (so $\delta^*\in\mathrm{Acc}_{\mathbb{G}(\delta^*)}$ since $\gamma\in M_m$) or $\delta\in K^{< 0}\gamma'\setminus K^{< 0}\gamma$, and therefore $\overline{\mathbb{FC}}(\delta)=-1$ and then we have $\delta<\vartheta\alpha$ and therefore, by the side inductive hypothesis, $\delta\in\mathrm{Acc}_m$. It follows that $\gamma'\in M_n$, and therefore $\vartheta\gamma'\in\mathrm{Acc}_m$ by the main inductive hypothesis.%
  
  Otherwise, there is some $\delta\in K^{< 0}\alpha$ with $\gamma\leq\delta$, and since $\delta\in\mathrm{Acc}_m$, also $\gamma\in\mathrm{Acc}_m$.
\end{proof}

\begin{theorem}\label{thm:ordinals_wf_idOmega}
  For all $\alpha$ and all $n\geq \mathbb{G}(\alpha^*)$, $\alpha^*\in\mathrm{Acc}_n$.
\end{theorem}
\begin{proof}
  By induction on terms. If $\alpha=\#\{\alpha_i\}$, $\alpha=\omega^{\alpha'}$, or $\alpha=\vartheta\alpha'$ then the claim follows from the inductive hypothesis and a lemma above. If $\alpha$ is $\Omega^{(0)}$, the claim holds because $\Omega^{(0)}$ is the smallest element of $M_n$.
\end{proof}

\subsection{Variables}\label{sec:poly_variables}

Finally, we should extend the system with a variable so that we can check the Key Lemma. We extend the notation system by variables $v^{(J)}$, where we interpret $v^{(J)}$ as a strongly critical ordinal term smaller than $\Omega^{(J)}$ and otherwise incomparable to strongly critical terms.

The definition of the substitution has to account for adjustments in levels creating by $\vartheta$.
\begin{definition}
  We define $\alpha[v\mapsto^{J}\beta]$ by induction on $\alpha$:
  \begin{itemize}
  \item $\#\{\alpha_i\}[v\mapsto^{J}\beta]=\#\{\alpha_i[v\mapsto^{J}\beta]\}$,
  \item $\omega^\alpha[v\mapsto^{J}\beta]=\omega^{\alpha[v\mapsto^{J}\beta]}$,
  \item $\Omega^{(J')}[v\mapsto^{J}\beta]=\Omega^{(J')}$,
  \item $(\vartheta\alpha)[v\mapsto^{J}\beta]=\vartheta(\alpha[v\mapsto^{J-1}\beta])$,
  \item $w^{(J')}[v\mapsto^{J}\beta]=
    \left\{\begin{array}{ll}
             \beta^{\leq 0}_{+J}&\text{if }v=w\\
             w^{(J')}&\text{if }v\neq w
           \end{array}\right.$.
       \end{itemize}
\end{definition}

We should only consider substituting into variables which appear precisely at the top level---that is, we want to substitute into things like $v^{(0)}\#\vartheta v^{(1)}$ but not $v^{(1)}\#\vartheta v^{(3)}$.
\begin{definition}
  $v$ is \emph{$J$-substitutable} in $\alpha$ if:
  \begin{itemize}
  \item $\alpha=\#\{\alpha_i\}$ and $v$ is $J$-substitutable in every $\alpha_i$,
  \item $\alpha=\omega^{\alpha'}$ and $v$ is $J$-substitutable in $\alpha'$,
  \item $\alpha=\Omega^{(J')}$,
  \item $\alpha=\vartheta\alpha'$ and $v$ is $J-1$-substitutable in $\alpha'$,
  \item $\alpha=w^{(J')}$ and either $v\neq w$ or $J=J'$.
  \end{itemize}
\end{definition}

\begin{definition}
  When $\overline{\mathbb{FC}}(\gamma)<0$, we define $D_{\gamma}(\beta)=D_{0,\gamma}(\beta)=\vartheta(\omega^{\Omega\#\beta}\#\gamma)$.  We define $D_{m+1,\gamma}(\beta)=D_{0,0}(D_{m,\gamma}(\beta))$.

  When $\overline{\mathbb{FC}}(\gamma)<0$, we define $\alpha\ll_\gamma\beta$ to hold if $\alpha<\beta$ and every $\eta\in K^{<0}\alpha$ is bounded by $D_{m,\gamma}(\beta)$ where $m$ is least so that $\overline{\mathbb{FC}}(D_{m,\gamma}(\beta))\leq\overline{\mathbb{FC}}(\eta)$.  
\end{definition}

\begin{lemma}[Key Lemma]\ 
  \begin{enumerate}
  \item If $\alpha<\beta$, $v$ is $0$-substitutable in $\alpha$ and $\beta$, and $\overline{\mathbb{FC}}(\gamma)<0$ then $\alpha[v\mapsto\gamma]<\beta[v\mapsto\gamma]$.
  \item If $\overline{\mathbb{FC}}(\delta)<0$ and $\alpha\ll_\delta\beta$ are ordinal terms so that every free variable is $J$-substitutable for some $J<0$ then $D_{\delta}(\alpha)\ll_0 D_\delta(\beta)$.
  \item If $\overline{\mathbb{FC}}(\delta)<0$ and $\alpha\ll_\delta\beta$, $\gamma\ll_\delta\beta$, and $v$ is $0$-substitutable in $\gamma$ then $D_{D_\delta(\alpha)^{\leq 0}_{-1}}(\gamma[v\mapsto^0 D_\delta(\alpha)^{\leq 0}_{-1}])\ll_0 D_\delta(\beta)$.
  \end{enumerate}
\end{lemma}
\begin{proof}
  \begin{enumerate}
  \item  By induction on $\alpha,\beta$, following the comparison between $\alpha$ and $\beta$. Any step involving $v$ is preserved when we replace $v$ with $\gamma$.

  \item  Almost immediate from the definition.
    
  \item  Since $\gamma<\beta$, we also have $\gamma[v\mapsto^0 D_\delta(\alpha)^{\leq 0}_{-1}]<\beta$ by (1), so $\omega^{\Omega\#\gamma[v\mapsto^0 D_\delta(\alpha)^{\leq 0}_{-1}]}\#D_\delta(\alpha)^{\leq 0}_{-1}<\omega^{\Omega\#\beta}\#\delta$. We have $K^{<0}(\omega^{\Omega\#\gamma[v\mapsto^0 D_\delta(\alpha)^{\leq 0}_{-1}]}\#D_\delta(\alpha)^{\leq 0}_{-1})\subseteq K^{<0}\gamma\cup K^{<0}\delta\cup \{D_\delta(\alpha)\}$, and so is bounded by $D_{m,\delta}(\beta)$ for the appropriate $m$.
  \end{enumerate}
\end{proof}

\subsection{The Path Not Taken}\label{sec:not_taken}

Before proceeding, we briefly discuss an alternative route one might take to develop the syntax of $\mathrm{OT}^{poly}_{\Omega_\omega}$, since it sometimes provides useful intuition.

In the alternate version, we avoid indices, having only a single cardinal term $\Omega$. We instead include variables from the beginning, and we adopt the rule that if $\alpha$ is an ordinal with free variables $v_1,\ldots,v_n$ then $\vartheta\alpha\beta_1\cdots\beta_n$ is an ordinal. The interpretation we have in mind is that this term represents $(\vartheta\alpha)[v_i\mapsto\beta_i]$.

Since we don't actually carry out the substitutions, we no longer need to worry about variable clashes---in the term $({\color{blue}\vartheta}{\color{blue}\Omega}\#v)(\Omega)$, there is no danger of thinking the black $\Omega$ is bound, since it is not even inside ${\color{blue}\vartheta}$.

We lose uniqueness, since we should have, for instance, $(\vartheta v)(\alpha\#\omega^\beta)=(\vartheta v\#\omega^w)(\alpha,\beta)$, but could easily recover this by adding some appropriate normal form.

If we followed through on this, we would run into the following obstacle---in order to compare $\vartheta\alpha\beta_1\cdots\beta_n$ with $\vartheta\alpha'\beta'_1\cdots\beta'_{n'}$, we would like to compare their \emph{uncollapsed} versions. But the uncollapsed version of $\vartheta\alpha\beta_1\cdots\beta_n$ cannot be $\alpha[v_i\mapsto\beta_i]$---this loses the distinction between $\Omega$'s at different levels. Rather, we want $\alpha$ to be something like $\lambda v_1\cdots v_n. \alpha$. But this means, for instance, that $(\lambda v.v)\Omega$ is fundamentally different that $\Omega$ (because the former collapses to something like $(\vartheta v)\Omega$ while the latter collapses to $\vartheta\Omega$). We could accept such issues, but it seems that the solution would end up replicating the counting and shifting we had to deal with above instead of rescuing us from it.

\section{Functions as Cardinals}
\subsection{Ordinal Terms}

We next describe a system introducing our new device, replacing $\Omega$ with a cardinal-like term $\Xi$, but $\Xi$ will have a \emph{function sort}---that is, we will have new terms of the form $\Xi(\alpha)$. Like $\Omega$, we will interpret $\Xi$ polymorphically, including expressions $\Xi^{(J)}(\alpha)$ for $J$ a natural number. Some device like polymorphism is necessary here---we will need to consider terms like $(v(\Xi(0)))[v\mapsto \vartheta w]$, which needs a way to ``protect'' the $\Xi(0)$ in the argument of $v$ from being collapsed when it is substituted into $\vartheta w$. As we will see, $\Xi$'s behavior as a function will lead us to a different definition of $K_\Xi$. Relatedly, it will be convenient to include some variables from the beginning.

\begin{definition}
  We define the ordinal terms $\mathrm{OT}_{\Xi}$ together with a distinguished subset $\mathbb{H}$ by:
  \begin{itemize}
  \item if $\{\alpha_0,\ldots,\alpha_{n-1}\}$ is a finite multi-set of elements of $\mathrm{OT}_{\Xi}$ in $\mathbb{H}$ and $n\neq 1$ then $\#\{\alpha_0,\ldots,\alpha_{n-1}\}$ is in $\mathrm{OT}_{\Xi}$,
  \item if $\alpha$ is in $\mathrm{OT}_{\Xi}$ then $\omega^\alpha$ is in $\mathbb{H}$,
  \item for each natural number $J$ and each $\alpha$ in $\mathrm{OT}_{\Xi}$, $\Xi^{(J)}(\alpha)$ is in $\mathbb{H}$,
  \item for any $\alpha$ in $\mathrm{OT}_{\Xi}$, $\vartheta\alpha$ is in $\mathbb{H}$.
  \item there are an infinite set of variables $V$ and, for each $v\in V$ and natural number $J$, $v^{(J)}$ is in $\mathrm{OT}_\Xi$.
  \end{itemize}

  We define $\mathbb{SC}$ to be all ordinal terms in $\mathbb{H}$ \emph{not} of the form $\omega^\alpha$.
\end{definition}

Our formal cardinalities are the same as in the previous section.

We can now define operations like shifting and substitution similar to the way we did above. There is one subtlety, which is that in $\Xi^{(J)}(\alpha)$, we wish everything in $\alpha$ to have cardinality at most\footnote{Explicitly, we need to prohibit something like ${\color{blue}\vartheta}\Xi({\color{blue}\Xi}(0))$, where the inner $\Xi$ is bound while the outer one is free. To help motivate why this expression is incorrect, consider the alternate syntax where we replace polymorphism with $\lambda$ expressions. We should be able to write this in the form $(\lambda v. s(v))(t)$ for some terms $s$ and $t$; as long as $t$ is an ordinal, we cannot do this.} that of $\Xi$. To enforce this syntactically, we take $\Xi^{(J)}(\Xi^{(0)}(0))$ to mean that the inner $\Xi^{(0)}$ is \emph{at the same level} as the outer one.

\begin{definition}
  We define $\alpha^{\leq J}_{\pm J'}$ by:
  \begin{itemize}
  \item $\#\{\alpha_i\}^{\leq J}_{\pm J'}=\#\{(\alpha_i) ^{\leq J}_{\pm J'}\}$,
  \item $(\omega^\alpha) ^{\leq J}_{\pm J'}=\omega^{\alpha^{\leq J}_{\pm J'}}$,
  \item $(\Xi^{(J'')}(\alpha))^{\leq J}_{\pm J'}=    \left\{\begin{array}{ll}
             \Xi^{(J''+J')}(\alpha)&\text{if }J''\leq J\\
             \Xi^{(J'')}(\alpha^{\leq J-J''}_{\pm J'})&\text{if }J''>J\\
           \end{array}\right.$,
  \item $(\vartheta\alpha)^{\leq J}_{\pm J'}=\vartheta(\alpha^{\leq J}_{\pm J'})$,
  \item $(v^{(J'')})^{\leq J}_{\pm J',\kappa'}=    \left\{\begin{array}{ll}
             v^{(J''+J')}&\text{if }J''\leq J\\
             v^{(J'')}&\text{if }J''> J\\
           \end{array}\right.$.
  \end{itemize}
\end{definition}

Note that, for this purpose, $v^{(J)}$ sits in the same cardinality as $\Xi^{(J)}$. (It will lie at the very bottom of this cardinality, so it will in other places behave like cardinality $J-1$ instead.)

Next we can define how substituting for the variables $v^{(J)}$ works.
\begin{definition}
  We define $\alpha[v\mapsto^{J}\beta]$ by induction on $\alpha$:
  \begin{itemize}
  \item $\#\{\alpha_i\}[v\mapsto^{J}\beta]=\#\{\alpha_i[v\mapsto^{J}\beta]\}$,
  \item $\omega^\alpha[v\mapsto^{J}\beta]=\omega^{\alpha[v\mapsto^{J}\beta]}$,
  \item $\Xi^{(J')}(\alpha)[v\mapsto^{J}\beta]=\Xi^{(J')}(\alpha[v\mapsto^{J-J'}\beta])$,
       \item $(\vartheta\alpha)[v\mapsto^{J}\beta]=\vartheta(\alpha[v\mapsto^{J-1}]\beta)$,
  \item $w^{(J')}[v\mapsto^{J}\beta]=\left\{\begin{array}{ll}
                                                      \beta^{\leq 0}_{+J}&\text{if }v=w\\
                                                      w&\text{otherwise}
                                                    \end{array}\right.$.
  \end{itemize}

\end{definition}

As before, we only want to substitute into $v$ when it is at the right level at the moment of substitution.

\begin{definition}
  We define when $v$ is $J$-substitutable in $\alpha$ by:
  \begin{itemize}
  \item $v$ is $J$-substitutable in $\#\{\alpha_i\}$ if $v$ is $J$-substitutable in each $\alpha_i$,
  \item $v$ is $J$-substitutable in $\omega^\alpha$ if $v$ is $J$-substitutable in $\alpha$,
  \item $v$ is $J$-substitutable in $\Xi^{(J')}(\alpha)$ if either $v$ does not appear in $\alpha$, or $J'>J'$ and $v$ is $J-J'$ substitutable in $\alpha$,
  \item $v$ is $J$-substitutable in $(\vartheta\alpha)$ if $v$ is $J-1$-substitutable in $\alpha$,
  \item $v$ is $J$-substitutable in $w^{(J')}$ if $v\neq w$ or $J=J'$.
  \end{itemize}
\end{definition}

We next collect up the cardinalities present in an ordinal term. %

\begin{definition}
   We define $\mathbb{FC}^{\leq J}(\alpha)$ inductively by:
  \begin{itemize}
  \item $\mathbb{FC}^{\leq J}(\#\{\alpha_i\})=\bigcup_i\mathbb{FC}^{\leq J}(\alpha_i)$,
  \item $\mathbb{FC}^{\leq J}(\omega^\alpha)=\mathbb{FC}^{\leq J}(\alpha)$,
  \item $\mathbb{FC}^{\leq J}(\Xi^{(J')}(\alpha))=\left\{\begin{array}{ll}
                                                      \{J'-J\}\cup \{J''+(J'-J)\mid J''\in \mathbb{FC}^{\leq 0}(\alpha)\}&\text{if }J'\leq J\\
                                                      \mathbb{FC}^{\leq J-J'}(\alpha)&\text{if }J'>J\\
                                                    \end{array}\right.$,
  \item $\mathbb{FC}^{\leq J}(\vartheta\alpha)=\mathbb{FC}^{\leq J-1}\alpha$,
  \item $\mathbb{FC}^{\leq J}(v^{(J')})=\left\{\begin{array}{ll}
                                                      J'-1&\text{if }J'\leq J\\
                                                      \emptyset&\text{if }J'>J\\
                                                    \end{array}\right.$.
  \end{itemize}

  We define $\overline{\mathbb{FC}}^{\leq J}(\alpha)=\sup\mathbb{FC}^{\leq J}(\alpha)$ (where $\sup\emptyset=-\infty$).
\end{definition}

The definition above takes $v^{(J)}$ to sit between the formal cardinality $(J,\Xi)$ and the next lower formal cardinality $(J+1,\Omega)$.

\begin{definition}
When $\alpha$ is an ordinal term, a \emph{parameter} of $\alpha$ is an ordinal term $\beta$ in $\mathbb{SC}$ so that there is some $\alpha'(v)$ with a fresh $0$-substitutable variable $v$ so that with $\alpha=\alpha'[v\mapsto^0\beta]$.

  The \emph{canonical abstraction} of $\alpha$ is the expresion $\alpha=\bar\alpha[v_1\mapsto^0\beta_1,\ldots,v_n\mapsto^0\beta_n]$ such that $\overline{\mathbb{FC}}^{\geq 0}(\bar\alpha)<0$ and each $\beta_i$ has the form $\Xi^{(0)}(\beta'_0)$.
\end{definition}
That is, the canonical abstraction is the expression we get by pulling out precisely the parameters which are of the form $\Xi^{(0)}(\beta'_0)$.
  
\begin{definition}
The definition of $K^{< J}\alpha$ contains a crucial difference:
  \begin{itemize}
  \item $K^{< J}\#\{\alpha_i\}=\bigcup_i K^{< J}\alpha_i$,
  \item $K^{< J}\omega^\alpha=K^{< J}\alpha$,
  \item $K^{< J}\Xi^{(J')}(\alpha)=\left\{\begin{array}{ll}
                                          \{\Xi^{J'-(J-1)}(\alpha)\}&\text{if }J'<J\\
                                          K^{\leq J-J'}\alpha&\text{if }J'\geq J
                                        \end{array}\right.$,
  \item $K^{< J}\vartheta\alpha=
    \left\{\begin{array}{ll}
             \left(\overline{(\vartheta\alpha)^{\leq 0}_{-J'}}\right)^{\leq 0}_{+1}&\text{if }\overline{\mathbb{FC}}^{\leq 0}(\vartheta\alpha)\leq J\\
             K^{\leq J,\kappa+\Xi}\alpha&\text{if }\overline{\mathbb{FC}}^{\leq 0}(\vartheta\alpha)>J
           \end{array}\right.$,
       \item $K^{< J} v^{(J')}=\left\{\begin{array}{ll}\{v^{(J-(J-1))}\}&\text{if }J'<J\\
                                                         \emptyset&\text{if }J'\geq J\end{array}\right.$.
  \end{itemize}
\end{definition}

In the definition of $K\vartheta\alpha$, where we include the case where $\overline{FC}^{\leq 0}(\vartheta\alpha)=J$. In this case, we think of $\vartheta\alpha$ as the result of applying a function to some $\Xi^{(0)}(\beta_i)$ term, and we place this \emph{function} in $K^{< J}$.

For instance, consider the ordinal $\vartheta\alpha={\color{red}\vartheta}{\color{blue}\vartheta}{\color{red}\Xi}^{(1)}(0)$. When we compare this to other ordinals, we will look at the set $K^{< 0}\alpha=\{\vartheta v\}$. That is, this ordinal recognizes that the ordinal $\alpha$ includes applying the \emph{function} $\vartheta v$ to $\Xi^{(0)}$; in our definition of $<$ below, we will reflect this by arranging that $\vartheta\alpha$ be closed under the function $\vartheta v$, in the sense that, for any $\beta<\vartheta\alpha$, $(\vartheta v)[v\mapsto\beta]<\vartheta\alpha$. (Note that $\vartheta\alpha$ itself is \emph{not} of this form, because the definition of substitution prevents anything in $\beta$ from getting bound during the substitution.)

This perhaps looks more conventional in the alternative notation---we could think of this ordinal as $\vartheta\left[(\lambda v. \vartheta v)\Xi(0)\right]$, which makes identifying $\lambda v. \vartheta v$ as an element of $K^{<0}\alpha$ a natural interpretation of the $K$ operator in this context. 

\begin{definition}
  We define the ordering $\alpha<\beta$:
  \begin{itemize}
\item $\#\{\alpha_i\}<\#\{\beta_j\}$ if there is some $\beta_{j_0}\in\{\beta_j\}\setminus\{\alpha_i\}$ such that, for all $\alpha_i\in\{\alpha_i\}\setminus\{\beta_j\}$, $\alpha_i<\beta_{j_0}$,
\item if $\beta\in\mathbb{H}$ then:
  \begin{itemize}
  \item $\beta<\#\{\alpha_i\}$ if there is some $i$ with $\beta\leq\alpha_i$,
  \item $\#\{\alpha_i\}<\beta$ if, for all $i$, $\alpha_i<\beta$,
  \end{itemize}
\item $\omega^\alpha<\omega^\beta$ if $\alpha<\beta$,
\item if $\beta\in\mathbb{SC}$ then:
  \begin{itemize}
  \item $\beta<\omega^\alpha$ if $\beta<\alpha$,
  \item $\omega^\alpha<\beta$ if $\alpha\leq\beta$,
  \end{itemize}
\item $\Xi^{(J)}(\alpha)<\Xi^{(J')}(\beta)$ if $J<J'$ or $J=J'$ and $\alpha<\beta$,
\item $\Xi^{(J)}(\alpha)<\vartheta\beta$ if there is a $\gamma\in K^{<0}\alpha$ with $\Xi^{(J)}(\alpha)\leq\gamma[v\mapsto^00]$,
\item $\vartheta\alpha<\vartheta\beta$ if:
  \begin{itemize}
  \item $\alpha<\beta$ and for any parameter $\beta'$ of $\beta$ and any $\gamma\in K^{< 0}\alpha$, $\gamma[v\mapsto^{0}\beta']<\vartheta\beta$,
  \item $\beta<\alpha$ and there is a parameter $\alpha'$ of $\alpha$ and a $\gamma\in K^{<0}\beta$ so that $\vartheta\alpha\leq\gamma[v\mapsto^0\alpha']$.
  \end{itemize}
\end{itemize}
\end{definition}

One might worry about the inductive definition involved in comparisons like $\gamma[v\mapsto^{0}\beta']<\vartheta\beta$. The following perspective illustrates why this is justified, as well as why it yields the strong fixed point property of $\vartheta\beta$ that we need.

Evaluate the inductive definition of $<$ between $\gamma$, with the variable, and $\vartheta\beta$. This is well-defined unless we need to compare a variable in $\gamma$ with a strongly critical subterm of $\vartheta\beta$. However, since could substitute $v$ with \emph{any} strongly critical subterm of $\vartheta\beta$, so we may assume $v$ is the largest strongly critical subterm of $\vartheta\beta$. In this way, we can complete the inductive comparison of $\gamma$ to $\vartheta\beta$.

Similarly, if $\gamma[v\mapsto^{0}\beta']<\vartheta\beta$ for all strongly critical subterms of $\vartheta\beta$, it follows that $\gamma[v\mapsto^{0}\delta]<\vartheta\beta$ for all $\delta<\vartheta\beta$, since replacing $v$ with such a $\delta$ does not change the result of the comparison.

\subsection{Well-Foundedness}

The basic structure of the proof is similar. We will classify ordinals based on their largest cardinal term---that is, given an ordinal $\alpha$ with $\overline{\mathbb{FC}}^{\leq 0}(\alpha)>-\infty$, we will shift it over so that $\overline{\mathbb{FC}}^{\leq 0}(\alpha^{*})=0$. We will further identify the largest $\beta$ so that $\Xi^{(0)}(\beta)$ appears $\alpha^*$.

\begin{definition}
  The \emph{fine cardinality} of $\alpha$ with $\overline{\mathbb{FC}}^{\leq 0}(\alpha)\in\{-\infty,0\}$ is $-\infty$ if $\overline{\mathbb{FC}}^{\leq 0}(\alpha)=-\infty$ and is $\beta$, where $\beta$ is maximal so that $\alpha=\alpha'[v\mapsto^0\Xi^{(0)}(\beta)]$ for some $\alpha'$ with $v$ appearing in $\alpha'$, if $\overline{\mathbb{FC}}^{\leq 0}(\alpha)=0$. We write $\kappa(\alpha)$ for the fine cardinality of $\alpha$.
\end{definition}

\begin{definition}
  We define a sequence of sets inductively:
  \begin{itemize}
  \item $M_{-\infty}$ is all ordinal terms with formal cardinality $-\infty$,
  \item given $M_c$, we let $\mathrm{Acc}_c$ be the well-founded part of $M_c$,
  \item for any $c=(n,\beta)$ so that $\beta=\beta^*\in\mathrm{Acc}_{(-\mathbb{G}(\beta^*),\kappa(\beta^*))}$, we let $M_c$ be the set of ordinal terms $\alpha$ with $\overline{FC}^{\leq 0}(\alpha)=0$, $\mathbb{G}(\alpha)\geq -n$, $\kappa(\alpha)=\beta$, and $\{\gamma^*\mid \gamma\in K^{<0}\alpha\}\subseteq\bigcup_{c'<c}\mathrm{Acc}_{c'}$
  \end{itemize}
  where $-\infty<(n,\beta)$ for all $n,\beta$ and $(n,\beta)<(n',\beta')$ if $n<n'$ or $n=n'$ and $\beta<\beta'$.
\end{definition}

As before, we have:
\begin{lemma}
  \begin{itemize}
  \item Whenever $c\leq c'$, $\alpha\in\mathrm{Acc}_c$, $\beta\in\mathrm{Acc}_{c'}$, we have $\alpha^{\geq 0}_{-k}\#\beta\in\mathrm{Acc}_{c'}$ where, if $c=(m,\alpha')$ and $c'=(n,\beta')$ then $k=n-m$ (and if $c$ does not have this form then $c=-\infty$ and $k$ does not matter).
  \item If $\alpha\in\mathrm{Acc}_c$ then $\omega^\alpha\in\mathrm{Acc}_c$.
  \item   If $\alpha\in\mathrm{Acc}_c$ then $\{(\vartheta\alpha)^*,(\vartheta(\alpha^{\leq 0}_{-1}))\}\subseteq\bigcup_{c'}\mathrm{Acc}_{c'}$.
  \end{itemize}
\end{lemma}

\subsection{Variables}

Finally, of course, we wish to verify the key lemma. We already have variables, but we need additional ones which are properly like the $\Xi$ terms---that is, we need variables which are functions.

So we extend our system by terms $V^{(J)}(\alpha)$ where $\alpha$ is an ordinal term and $V$ is our new variable and $V^{(J)}(\alpha)$ is below $\Xi^{(J)}(\alpha)$. As usual, we do not allow $V$ to appear inside $\vartheta$.

\begin{definition}
  When $\alpha$ are ordinal terms and $\beta$ is an ordinal term with a distinguished (ordinal) variable $w$ which is $0$-substitutable in $\beta$, we define $\alpha[V\mapsto^J\beta(w)]$ by induction on $\alpha$:
  \begin{itemize}
  \item $\#\{\alpha_i\}[V\mapsto^J\beta(w)]=\#\{\alpha_i[V\mapsto^J\beta(w)]\}$,
  \item $\omega^\alpha[V\mapsto^J\beta(w)]=\omega^{\alpha[V\mapsto^J\beta(w)]}$,
  \item $\Xi^{(J')}(\alpha)[V\mapsto^J\beta(w)]=\Xi^{(J')}([V\mapsto^{J-J'}\beta(w)])$,
  \item $(\vartheta\alpha)[V\mapsto^J\beta(w)]=\vartheta(\alpha[V\mapsto^{J-1}\beta(w)])$,
  \item $V^{(J')}(\alpha)[V\mapsto^J\beta(w)]=\beta(\alpha[V\mapsto^0\beta(w)])$.
  \end{itemize}

  We define when $V$ is \emph{$J$-substitutable} in $\alpha$ by:
  \begin{itemize}
  \item $V$ is $J$-substitutable in $\#\{\alpha_i\}$ if $V$ is $J$-substitutable in each $\alpha_i$,
  \item $V$ is $J$-substitutable in $\omega^\alpha$ if $V$ is $J$-substitutable in $\alpha$,
  \item $V$ is $J$-substitutable in $\Xi^{(J')}(\alpha)$ if $J\leq J'$ and $V$ is $J-J'$-substitutable in $\alpha$,
  \item $V$ is $J$-substitutable in $\vartheta\alpha$ if $V$ does not appear in $\vartheta\alpha$,%
  \item $V$ is $J$-substitutable in $V^{(J')}(\alpha)$ if $J=J'$ and $V$ is $0$-substitutable in $\alpha$.
  \end{itemize}
\end{definition}

\begin{definition}
  When $\overline{\mathbb{FC}}(\gamma)<0$, we define $D_\gamma(\beta)=D_{0,\gamma}(\beta)=\vartheta(\omega^{\Xi^{(0)}(1)\#\beta}\#\gamma(\Xi^{(0)}(0)))$ and $D_{m+1,\gamma}(\beta)=D_{0,0}(D_{m,\gamma}(\beta))$.

  When $\overline{\mathbb{FC}}(\gamma)<0$, we define $\alpha\ll_\gamma\beta$ to hold if $\alpha<\beta$ and every $\eta\in K^{<0}\alpha$ is bounded by $D_{m,\gamma}(\beta)$ where $m$ is least so that $\overline{\mathbb{FC}}(\gamma)\leq\overline{\mathbb{FC}}(\eta)$.
\end{definition}
We expect that $\eta$ has free variables; we understand this to be the comparison in which the free variables of $\eta$ are larger than any proper strongly critical subterm of $D_{m,\gamma}(\beta)$.

\begin{lemma}[Key Lemma for $\Xi$]
  \begin{enumerate}
  \item If $\alpha<\beta$, and $v$ is $0$-substitutable in $\gamma$ then $\gamma[v\mapsto^0\alpha]<\gamma[v\mapsto^0\beta]$.
  \item If $\overline{\mathbb{FC}}(\delta)<0$, $\alpha\ll_\delta\beta$, $V$ is $0$-substitutable in $\alpha$ and $\beta$, $\overline{\mathbb{FC}}(\gamma)<0$, and $w$ is $0$-substitutable in $\gamma$ then $\alpha[V\mapsto^{0}\gamma(w)]<_\Omega\beta[V\mapsto^{0}\gamma(w)]$.
  \item If $\overline{\mathbb{FC}}(\delta)<0$ and $\alpha\ll_\delta\beta$ are ordinal terms so that every free variable is $J$-substitutable for some $J<0$ then $D_\delta(\alpha)\ll_0 D_\delta(\beta)$.
  \item If $\overline{\mathbb{FC}}(\delta)<0$, $\alpha\ll_\delta\beta$, $\gamma\ll_\delta\beta$, $V$ is $0$-substitutable in $\gamma$, $V$ does not appear in $\beta$, $w$ is $0$-substitutable in $\alpha$, then $D_{0,D_{0,\delta}(\alpha)^{\leq 0}_{-1}}(\gamma[V\mapsto^0 D_{0,\delta}(\alpha)^{\leq 0}_{-1}])\ll_0D_{0,\delta}(\alpha)$.
  \end{enumerate}
\end{lemma}
\begin{proof}
  \begin{enumerate}
  \item By induction on $\gamma$.
  \item By induction on the comparison between $\alpha$ and $\beta$. The case which would otherwise be complicated is avoided since $V$ does not appear inside any $\vartheta$.
  \item Since $\alpha<\beta$ we have $\omega^{\Xi^{(0)}(1)\#\alpha}\#\delta(\Xi^{(0)}(0))<\omega^{\Xi^{(0)}(1)\#\beta}\#\delta(\Xi^{(0)}(0))$. If $\eta\in K^{<0}(\omega^{\Xi^{(0)}(1)\#\alpha}\#\delta(\Xi^{(0)}(0)))$ then $\eta\in K^{<0}(\alpha)\cup \{\delta\}$ and is therefore bounded by $D_\gamma(\beta)$, and furthermore by $D_{m,\gamma}(\beta)$ for the suitable $m$.
    
  \item We have $\omega^{\Xi^{(0)}(1)\#\gamma[V\mapsto^0 D_{0,\delta}(\alpha)^{\leq 0}_{-1}]}\# D_{0,\delta}(\alpha)^{\leq 0}_{-1}(\Xi^{(0)}(0))<\omega^{\Xi^{(0)}(1)\#\beta}\# D_{0,\delta}(\alpha)^{\leq 0}_{-1}(\Xi^{(0)}(0))$ (using (1) to show that $\gamma[V\mapsto^0 D_{0,\delta}(\alpha)^{\leq 0}_{-1}]<\beta$). Suppose $\eta\in K^{<0}(\omega^{\Xi^{(0)}(1)\#\gamma[V\mapsto^0 D_{0,\delta}(\alpha)^{\leq 0}_{-1}]}\# D_{0,\delta}(\alpha)^{\leq 0}_{-1}(\Xi^{(0)}(0)))$. Then $\eta\in K^{<0}\gamma\cup\{D_{0,\delta}(\alpha)^{\leq 0}_{-1}(w)\}$ and is therefore bounded by the suitable $D_{m,\gamma}(\beta)$.    
  \end{enumerate}
\end{proof}

\section{Towards $\Pi^1_2\mhyphen\mathrm{CA}_0$}
\subsection{Ordinal Terms}

Just as we can place a cardinal $\Omega_2$ above $\Omega_1$, we can stack up polymorphic symbols any way we like, as well as mixing these polymorphic symbols with conventional ones---we could have an ordinary cardinal $\Omega_1$, then a polymorphic symbol $\Omega$ placed above $\Omega_1$, then a polymorphic $\Xi$ placed above that, and then another cardinal $\Omega_2$ above all of those.

In this section, we wish to consider a system with cardinals in the form
\[\Omega_1<\Omega_2<\cdots<\Omega_n<\cdots<\Xi<\Omega_{\Omega+1}<\cdots<\Omega_{\Omega+n}<\cdots.\]
(We write the second batch of $\Omega$ cardinals with indices $\Omega+n$, rather than $\omega+n$, to emphasize that they are much larger than the cardinal usually written $\Omega_{\omega+n}$; this is perhaps a bit misleading because they are \emph{also} much larger than the cardinal written $\Omega_{\Omega_1+1}$.)

The lower $\Omega$ cardinals $\Omega_n$ will behave conventionally. However we will do something more complicated with the upper $\Omega$ cardinals. The polymorphic $\Xi$ term introduces a sequence of cardinalities $\Xi^{(0)}>\Xi^{(-1)}>\cdots$, and we choose to have the $\Omega_{\Omega+n}$ cardinals get ``caught'' in this loop---that is, we will have $\Omega_{\Omega+n}^{(0)}>\cdots>\Omega_{\Omega+1}^{(0)}>\Xi^{(0)}>\cdots>\Omega_{\Omega+n}^{(-1)}>\cdots$.

Having specified that, the construction of the terms is mostly a matter of assembling ideas from above in one place.
\begin{definition}
  We define the ordinal terms $\mathrm{OT}_{\omega+\Xi+\omega}$ together with a distinguished subset $\mathbb{H}$ by:
  \begin{itemize}
  \item if $\{\alpha_0,\ldots,\alpha_{n-1}\}$ is a finite multi-set of elements of $\mathrm{OT}_{\omega+\Xi+\omega}$ in $\mathbb{H}$ and $n\neq 1$ then $\#\{\alpha_0,\ldots,\alpha_{n-1}\}$ is in $\mathrm{OT}_{\omega+\Xi+\omega}$,
  \item if $\alpha$ is in $\mathrm{OT}_{\omega+\Xi+\omega}$ then $\omega^\alpha$ is in $\mathbb{H}$,
  \item for each $n>0$, $\Omega_n$ is in $\mathbb{H}$,
  \item for each $J\leq 0$ and each $n>0$, $\Omega^{(J)}_{\Omega+n}$ is in $\mathbb{H}$,
  \item for each $J\leq 0$ and each $\alpha$ in $\mathrm{OT}_{\omega+\Xi+\omega}$, $\Xi^{(J)}(\alpha)$ is in $\mathbb{H}$,
  \item for any $\alpha$ in $\mathrm{OT}_{\omega+\Xi+\omega}$, $\vartheta_{\Omega_n}\alpha$ is in $\mathbb{H}$,
  \item for any $\alpha$ in $\mathrm{OT}_{\omega+\Xi+\omega}$, $\vartheta_{\Omega_{\Omega+n}}\alpha$ is in $\mathbb{H}$,
  \item for any $\alpha$ in $\mathrm{OT}_{\omega+\Xi+\omega}$, $\vartheta_\Xi\alpha$ is in $\mathbb{H}$.
  \item there are an infinite set of variables $V$ and, for each $v\in V$ and non-positive integer $J$, $v^{(J)}$ is in $\mathrm{OT}_{\omega+\Xi+\omega}$.
  \end{itemize}

  We define $\mathbb{SC}$ to be all ordinal terms in $\mathbb{H}$ \emph{not} of the form $\omega^\alpha$ and $\mathbb{C}$ to be ordinal terms of the form $\Omega_n$, $\Omega^{(J)}_{\Omega+n}$, $\Xi^{(J)}(\alpha)$, or $v^{(J)}$.
\end{definition}

Our formal cardinalities are a bit more complicated now. They look like:
\begin{align*}
  -\infty<1<2<\cdots<n<&\cdots<(J-1,0)\\
                <(J-1,1)<&\cdots<(J-1,n)<\cdots<(J-1,\infty)<(J,0)<\\&\cdots<(J,n)<\cdots<(0,\infty).\end{align*}
Here $-\infty$ is the smallest cardinality, for countable ordinals. $n$ is the cardinality of $\Omega_n$. $(J,0)$ is the cardinality of $\Xi^{(J)}$ and $(J,n)$ is the cardinality of $\Omega^{(J)}_{\Omega+n}$. Nothing sits at the cardinality $(J,\infty)$---rather, it denotes the state in which $\Xi^{(J+1)}$ is bound, but no $\Omega_n^{(J)}$ has been bound yet.

\begin{definition}
  We define the \emph{formal cardinalities} to the
  \[\{-\infty\}\cup\{n\mid n\in\mathbb{N}, n>0\}\cup\{(J,m)\mid J\leq 0,m\in\mathbb{N}\cup\{\infty\}\}\]
  with the ordering described above.

  We call a formal cardinality \emph{large} if it has the form $(J,m)$.
\end{definition}

We need to define a few operations on formal cardinalities we will use regularly.
\begin{definition}
  We define $(J',n)-J=(J'-J,\infty)$, $\min\{(J,m),n\}=(J,\min\{m,n\})$, and $J'-(J,n)=J'-J$.
\end{definition}

\begin{definition}
Let $c$ be a large formal cardinality and $J$ an integer. We define $\alpha^{< c}_{\pm J}$ by:
  \begin{itemize}
  \item $\#\{\alpha_i\}^{< c}_{\pm J}=\#\{(\alpha_i) ^{< c}_{\pm J}\}$,
  \item $(\omega^\alpha) ^{<c}_{\pm J}=\omega^{\alpha^{<c}_{\pm J}}$,
  \item $(\Omega_n)^{<c}_{\pm J}=\Omega_n$,
  \item $(\Omega^{(J')}_n)^{<c}_{\pm J}=
    \left\{\begin{array}{ll}
             \Omega^{(J'+J)}_n&\text{if }(J',n)<c\\
             \Omega^{(J')}_n&\text{if }(J',n)\geq c\\
           \end{array}\right.$,
  \item $(\Xi^{(J')}(\alpha))^{< c}_{\pm J}=    \left\{\begin{array}{ll}
             \Xi^{(J'+J)}(\alpha)&\text{if }(J',0)<c\\
             \Xi^{(J')}(\alpha^{< c-J'}_{\pm J})&\text{if }(J',0)\geq c\\
           \end{array}\right.$,
  \item $(\vartheta_{\Omega_n}\alpha)^{<c}_{\pm J}=\vartheta_{\Omega_n}\alpha$,
  \item $(\vartheta_{\Omega_{\Omega+n}}\alpha)^{<c}_{\pm J}=\vartheta_{\Omega_{\Omega+n}}(\alpha^{<\min\{c,n\}}_{\pm J})$,
  \item $(\vartheta_\Xi\alpha)^{<c}_{\pm J}=\vartheta_\Xi(\alpha^{<c-1}_{\pm J})$,
  \item $(v^{(J')})^{<c}_{\pm J}=    \left\{\begin{array}{ll}
             v^{(J'+J)}&\text{if }(J',0)<c\\
             v^{(J')}&\text{if }(J',0)\geq c\\
           \end{array}\right.$.
  \end{itemize}
\end{definition}

Note that, for this purpose, $v^{(J)}$ sits in the same cardinality as $\Xi^{(J)}$.

Next we can define how substituting for the variables $v^{(J)}$ works.
\begin{definition}
  Let $J\leq 0$. We define $\alpha[v\mapsto^{J}\beta]$ by induction on $\alpha$:
  \begin{itemize}
  \item $\#\{\alpha_i\}[v\mapsto^{J}\beta]=\#\{\alpha_i[v\mapsto^{J}\beta]\}$,
  \item $\omega^\alpha[v\mapsto^{J}\beta]=\omega^{\alpha[v\mapsto^{J}\beta]}$,
  \item $\Omega_n[v\mapsto^J\beta]=\Omega_n$,
  \item $\Omega^{(J')}_{\Omega+n}[v\mapsto^J\beta]=\Omega^{(J')}_n$,
  \item $\Xi^{(J')}(\alpha)[v\mapsto^{J}\beta]=
    \left\{\begin{array}{ll}
             \Xi^{(J')}(\alpha[v\mapsto^{J-J'}\beta])&\text{if }J\leq J'\\
             \Xi^{(J')}(\alpha)&\text{if }J>J'\\
           \end{array}\right.$
       \item $(\vartheta_{\Omega_n}\alpha)[v\mapsto^{J}\beta]=\vartheta_{\Omega_n}\alpha$,
       \item $(\vartheta_{\Omega_{\Omega+n}}\alpha) [v\mapsto^{J}\beta]=\vartheta_{\Omega_{\Omega+n}}\alpha[v\mapsto^{J}\beta]$,
       \item $(\vartheta_\Xi\alpha)[v\mapsto^{J}\beta]=\vartheta_\Xi(\alpha[v\mapsto^{J-1}]\beta)$,
  \item $w^{(J')}[v\mapsto^{J}\beta]=\left\{\begin{array}{ll}
                                              \beta^{< (0,\infty)}_{+J}&\text{if }v=w\\
                                              w&\text{otherwise}
                                            \end{array}\right.$.
  \end{itemize}

\end{definition}

As before, we only want to substitute into $v$ when it is at the right level at the moment of substitution.

\begin{definition}
  We define when $v$ is $J$-substitutable in $\alpha$ by:
  \begin{itemize}
  \item $v$ is $J$-substitutable in $\#\{\alpha_i\}$ if $v$ is $J$-substitutable in each $\alpha_i$,
  \item $v$ is $J$-substitutable in $\omega^\alpha$ if $v$ is $J$-substitutable in $\alpha$,
  \item $v$ is always $J$-substitutable in $\Omega_n$,
  \item $v$ is always $J$-substitutable in $\Omega^{(J')}_{\Omega+n}$,
  \item $v$ is $J$-substitutable in $\Xi^{(J')}(\alpha)$ if either $v$ does not appear in $\alpha$ or $J\leq J'$ and $v$ is $J-J'$-substitutable in $\alpha$,
  \item $v$ is $J$-substitutable in $(\vartheta_{\Omega_n}\alpha)$,
  \item $v$ is $J$-substitutable in $(\vartheta_{\Omega_{\Omega+n}}\alpha)$ if $v$ is $J$-substitutable in $\alpha$,
  \item $v$ is $J$-substitutable in $(\vartheta_\Xi\alpha)$ if $v$ is $J-1$-substitutable in $\alpha$,
  \item $v$ is $c$-substitutable in $w^{(J')}$ if $v\neq w$ or $J=J'$.
  \end{itemize}
\end{definition}

We next collect up the cardinalities present in an ordinal term. %

\begin{definition}
  Let $c$ be a large formal cardinality. We define $\mathbb{FC}^{<c}(\alpha)$ inductively by:
  \begin{itemize}
  \item $\mathbb{FC}^{<c}(\#\{\alpha_i\})=\bigcup_i\mathbb{FC}^{<c}(\alpha_i)$,
  \item $\mathbb{FC}^{<c}(\omega^\alpha)=\mathbb{FC}^{<c}(\alpha)$,
  \item $\mathbb{FC}^{<c}(\Omega_n)=\{n\}$,
  \item $\mathbb{FC}^{<c}(\Omega^{(J')}_{\Omega+n})=\left\{\begin{array}{ll}
                                                      \{(J'-c,n)\}&\text{if }(J',n)<c\\
                                                      \emptyset&\text{if }(J',n)\geq c
                                                    \end{array}\right.$
  \item $\mathbb{FC}^{< c}(\Xi^{(J')}(\alpha))=\left\{\begin{array}{ll}
                                                        \{(J'-c,0)\}\cup \mathbb{FC}^{<c-J'}(\alpha)&\text{if }(J',0)<c\\
                                                        \mathbb{FC}^{<c-J'}(\alpha)&\text{if }(J',0)\geq c
                                                    \end{array}\right.$,
  \item $\mathbb{FC}^{<c}(\vartheta_{\Omega_n}\alpha)=\mathbb{FC}^{<c}\alpha$,
  \item $\mathbb{FC}^{<c}(\vartheta_{\Omega_{\Omega+n}}\alpha)=\mathbb{FC}^{<\min\{c,n\}}\alpha$,
  \item $\mathbb{FC}^{<c}(\vartheta_\Xi\alpha)=\mathbb{FC}^{<c-1}\alpha$,
  \item $\mathbb{FC}^{<c}(v^{(J')})=\left\{\begin{array}{ll}
                                             \{(J'-c,0)\}&\text{if }(J',0)<c\\
                                             \emptyset&\text{if }(J',0)\geq c
                                           \end{array}\right.$.
  \end{itemize}

  We define $\overline{\mathbb{FC}}^{<c}(\alpha)=\sup\mathbb{FC}^{<c}(\alpha)$ (where $\sup\emptyset=-\infty$).
\end{definition}

\begin{definition}
We now define the critical subterms for each type of cardinal.
  
  \begin{itemize}
  \item $K_n\#\{\alpha_i\}=\bigcup_i K_n\alpha_i$,
  \item $K_n\omega^\alpha=K_n\alpha$,
  \item $K_n\Omega_m=\left\{\begin{array}{ll}
                              \{\Omega_m\}&\text{if }m<n\\
                              \emptyset&\text{if }m\geq n
                              \end{array}\right.$,
  \item $K_n\Omega^{(J')}_m=\emptyset$,
  \item $K_n\Xi^{(J')}(\alpha)=\emptyset$,
  \item $K_n\vartheta_{\Omega_m}\alpha=\left\{\begin{array}{ll}
                                                \{\vartheta_{\Omega_m}\alpha\}&\text{if }m\leq n\\
                                                K_n\alpha&\text{if }m>n\\
                                              \end{array}\right.$,
  \item $K_n\vartheta_{\Omega_{\Omega+m}}\alpha=K_n\alpha$,
  \item $K_n\vartheta_\Xi\alpha=K_n\alpha$,
  \item $K^{<c}_{\Omega+n} v^{(J')}=\emptyset$.
 \end{itemize}

  \begin{itemize}
  \item $K^{<c}_{\Omega+n}\#\{\alpha_i\}=\bigcup_i K^{<c}_{\Omega+n}\alpha_i$,
  \item $K^{<c}_{\Omega+n}\omega^\alpha=K^{<c}_{\Omega+n}\alpha$,
  \item $K^{<c}_{\Omega+n}\Omega_m=\{\Omega_m\}$,
  \item $K^{<c}_{\Omega+n}\Omega^{(J')}_m=\left\{\begin{array}{ll}
                                          \{\Omega^{(J'-c)}_m\}&\text{if }(J',m)<c\\
                                          \emptyset&\text{if }(J',m)\geq c
                                        \end{array}\right.$,
  \item $K^{<c}_{\Omega+n}\Xi^{(J')}(\alpha)=\left\{\begin{array}{ll}
                                          \{\Xi^{(J'-c)}(\alpha)\}&\text{if }(J',0)<c\\
                                          K^{<c-J'}_{\Omega+n}\alpha&\text{if }(J',0)\geq c
                                        \end{array}\right.$,
  \item $K^{<c}_{\Omega+n}\vartheta_{\Omega_m}\alpha=\vartheta_m\alpha$,
  \item $K^{<c}_{\Omega+n}\vartheta_{\Omega_{\Omega+m}}\alpha=\left\{\begin{array}{ll}
                                            (\vartheta_{\Omega_{\Omega+m}}\alpha)^{< (0,\infty)}_{+0-c}&\text{if }\overline{\mathbb{FC}}^{<(0,\infty)}(\vartheta_{\Omega_{\Omega+m}}\alpha)< \min\{c,n\}\\
                                            K^{<\min\{c,n\}}_{\Omega+n}\alpha&\text{if }\overline{\mathbb{FC}}^{< (0,\infty)}(\vartheta_{\Omega_{\Omega+m}}\alpha)\geq \min\{c,n\}
                                          \end{array}\right.$,
  \item $K^{<c}_{\Omega+n}\vartheta_\Xi\alpha=\left\{\begin{array}{ll}
                                            (\vartheta_\Xi\alpha)^{< (0,\infty)}_{+0-c}&\text{if }\overline{\mathbb{FC}}^{<(0,\infty)}(\vartheta_\Xi\alpha)<\min\{c,n\}\\
                                            K^{<c-1}_\Omega\alpha&\text{if }\overline{\mathbb{FC}}^{<(0,\infty)}(\vartheta_\Xi\alpha)\geq\min\{c,n\}
                                          \end{array}\right.$,
  \item $K^{<c}_{\Omega+n} v^{(J')}=\left\{\begin{array}{ll}\{v^{(J'-c)}\}&\text{if }(J',0)< \min\{c,n\}\\
                                                                                        \emptyset&\text{if }(J',0)>\min\{c,n\}\end{array}\right.$.
 \end{itemize}

  \begin{itemize}
  \item $K^{<c}_\Xi\#\{\alpha_i\}=\bigcup_i K^{<c}_\Xi\alpha_i$,
  \item $K^{<c}_\Xi\omega^\alpha=K^{<c}_\Xi\alpha$,
  \item $K^{<c}_\Xi\Omega_n=\{\Omega_n\}$,
  \item $K^{<c}_\Xi\Omega_m^{(J')}=\left\{\begin{array}{ll}
                                          \{\Omega^{(J'-c)}_m\}&\text{if }(J',m)<c\\
                                          \emptyset&\text{if }(J',m)\geq c
                                        \end{array}\right.$,
  \item $K^{<c}_\Xi\Xi^{(J')}(\alpha)=\left\{\begin{array}{ll}
                                          \{(\Xi^{J'}(\alpha))^{<(0,\infty)}_{+0-c}\}&\text{if }(J',0)<c\\
                                          K^{< c-J'}_\Xi\alpha&\text{if }(J',0)\geq c
                                        \end{array}\right.$,
  \item $K^{<c}_\Xi\vartheta_{\Omega_n}\alpha=\{\vartheta_{\Omega_n}\alpha\}$,
  \item $K^{<c}_\Xi\vartheta_{\Omega_{\Omega+n}}\alpha=
    \left\{\begin{array}{ll}
             (\vartheta_{\Omega_{\Omega+n}}\alpha)^{<(0,\infty)}_{+0-(c-1)}&\text{if }\overline{\mathbb{FC}}^{<(0,\infty)}(\vartheta_{\Omega_{\Omega+n}}\alpha)<c-1\\
             K_\Xi^{<\min\{c,n\}}\alpha&\text{if }\overline{\mathbb{FC}}^{<(0,\infty)}(\vartheta_{\Omega_{\Omega+n}}\alpha)\geq c-1
           \end{array}\right.$,
  \item $K^{<c}_\Xi\vartheta_\Xi\alpha=
    \left\{\begin{array}{ll}
             \left(\overline{(\vartheta_\Xi\alpha)^{<(0,\infty)}_{+1}}\right)^{<(0,\infty)}_{+0-c}&\text{if }\overline{\mathbb{FC}}^{< (0,\infty)}(\vartheta_\Xi\alpha)<c\\
             K_\Xi^{<c-1}\alpha&\text{if }\overline{\mathbb{FC}}^{<(0,\infty)}(\vartheta_\Xi\alpha)\geq c
           \end{array}\right.$,
       \item $K^{<c}_\Xi v^{(J')}=\left\{\begin{array}{ll}\{v^{(J'-c)}\}&\text{if }(J',0)< c\\
                                                         \emptyset&\text{if }(J',0)\geq c\end{array}\right.$.
  \end{itemize}
\end{definition}

\begin{definition}
  We define the ordering $\alpha<\beta$:
  \begin{itemize}
  \item $\#\{\alpha_i\}<\#\{\beta_j\}$ if there is some $\beta_{j_0}\in\{\beta_j\}\setminus\{\alpha_i\}$ such that, for all $\alpha_i\in\{\alpha_i\}\setminus\{\beta_j\}$, $\alpha_i<\beta_{j_0}$,
\item if $\beta\in\mathbb{H}$ then:
  \begin{itemize}
  \item $\beta<_\kappa\#\{\alpha_i\}$ if there is some $i$ with $\beta\leq\alpha_i$,
  \item $\#\{\alpha_i\}<\beta$ if, for all $i$, $\alpha_i<\beta$,
  \end{itemize}
\item $\omega^\alpha<\omega^\beta$ if $\alpha<\beta$,
\item if $\beta\in\mathbb{SC}$ then:
  \begin{itemize}
  \item $\beta<\omega^\alpha$ if $\beta<\alpha$,
  \item $\omega^\alpha<\beta$ if $\alpha\leq\beta$,
  \end{itemize}
\item $\Omega_m<\Omega^{(J)}_{\Omega+n}$,
\item $\Omega_m<\Xi^{(J)}(\alpha)$,
\item $\Omega^{(J)}_{\Omega+m}<\Omega^{(J')}_{\Omega+n}$ if $J<J'$ or $J=J'$ and $m<n$,
\item $\Xi^{(J)}(\alpha)<\Omega^{(J')}$ if $J\leq J'$,
\item $\Omega^{(J)}_{\Omega+n}<\Xi^{(J')}(\alpha)$ if $J<J'$,
\item $\Xi^{(J)}(\alpha)<\Xi^{(J')}(\beta)$ if $J<J'$ or $J=J'$ and $\alpha<\beta$,
\item if $\alpha\in\mathbb{C}$ then $\alpha<\vartheta_{\Omega_n}\beta$ if there is a $\gamma\in K_n\beta$ with $\alpha\leq\gamma$,
\item if $\alpha\in\mathbb{C}$ then $\alpha<\vartheta_{\Omega_{\Omega+n}}\beta$ if there is a $\gamma\in K^{<(0,n)}_{\Omega+n}\beta$ with $\alpha\leq\gamma$,
\item if $\alpha\in\mathbb{C}$ then $\alpha<\vartheta_{\Xi}\beta$ if there is a $\gamma\in K^{<(0,0)}_{\Xi}\beta$ with $\alpha\leq\gamma$,
\item to compare $\vartheta_c\alpha$ and $\vartheta_{c'}\beta$ we first define the critical sets $C,D$:
  \begin{itemize}
  \item if $c\neq \Xi$ then $C=K^{<(0,c)}_c\alpha$,
  \item if $c'\neq \Xi$ then $D=K^{<(0,c)}_{c'}\beta$,
  \item if $c=\Xi$ then $C$ is the set of $\gamma[v\mapsto\beta']$ such that $\gamma\in K^{<(0,0)}_\Xi\alpha$ and $\beta'$ is a parameter of $\beta$,
  \item if $c'=\Xi$ then $D$ is the set of $\gamma[v\mapsto\alpha']$ such that $\gamma\in K^{<(0,0)}_\Xi\beta$ and $\alpha'$ is a parameter of $\alpha$.
  \end{itemize}
  we then set $\vartheta_c\alpha<\vartheta_{c'}\beta$ if either
  \begin{itemize}
  \item there is a $\delta\in D$ with $\vartheta_c\alpha\leq\delta$, 
  \item there is no $\delta\in C$ $\vartheta_c\beta\leq\delta$ and either $c=c'$ and $\alpha<\beta$, or $c<c'$.
  \end{itemize}
\end{itemize}
\end{definition}

We can now introduce variables at any cardinality---ordinal variables for any $\Omega_n$ or $\Omega_{\Omega+n}$, or function variables for $\Xi$---and establish corresponding versions of the Key Lemma as above.

\printbibliography
\end{document}